\documentclass[12pt]{article}

\usepackage{authblk}
\author{Gustavo de Paula Ramos\thanks{gpramos@icmc.usp.br}}
\affil{
Instituto de Ciências Matemáticas e de Computação,
Universidade de São Paulo,
Avenida Trabalhador São-Carlense 400,
13566-590,
São Carlos - SP,
Brazil
}

\usepackage{amsmath}
\numberwithin{equation}{section}

\usepackage{amsthm}
\newtheorem{thm}{Theorem}[section]
\newtheorem{prop}[thm]{Proposition} 
\newtheorem{lem}[thm]{Lemma}

\theoremstyle{definition}

\newtheorem{question}[thm]{Question}
\newtheorem*{ex}{Example}

\theoremstyle{remark}

\usepackage{amsfonts, mathrsfs}

\usepackage[
	backend = biber,
	sorting = none,
	url = false
]{biblatex}
\AtEveryBibitem{\clearlist{language}}
\title{Ground states of the Schrödinger equation coupled with fourth-order gravitation -- Part 1: the case $K_{a, b} \leq 0$}

\usepackage{mathtools}
\usepackage[shortlabels]{enumitem}
\newcommand{\nat}{\mathbb{N}}
\newcommand{\real}{\mathbb{R}}

\newcommand{\iu}{\mathrm{i}}
\newcommand{\eps}{\varepsilon}

\newcommand{\A}{\mathcal{A}}
\newcommand{\D}{\mathcal{D}}
\newcommand{\E}{\mathcal{E}}
\newcommand{\K}{\mathcal{K}}

\newcommand{\V}{\mathcal{V}}

\newcommand{\dif}{\ \mathrm{d}}
\newcommand{\ddif}{\mathrm{d}}

\newcommand{\Sobolev}{\mathscr{H}}
\newcommand{\Lebesgue}{\mathscr{L}}

\newcommand{\Sphere}{\mathscr{S}}

\DeclarePairedDelimiter{\abs}{\lvert}{\rvert}
\DeclarePairedDelimiter{\norm}{\lVert}{\rVert}
\DeclarePairedDelimiter{\parens}{(}{)}
\DeclarePairedDelimiter{\set}{\{}{\}}

\DeclarePairedDelimiter{\cci}{\lbrack}{\rbrack}
\DeclarePairedDelimiter{\coi}{\lbrack}{\lbrack}
\DeclarePairedDelimiter{\oci}{\rbrack}{\rbrack}
\DeclarePairedDelimiter{\ooi}{\rbrack}{\lbrack}

\begin{document}
\maketitle

\begin{abstract}
We are interested in the existence and asymptotic behavior of ground states of the following normalized nonlocal semilinear problem:
\[
\begin{cases}
- \Delta u + (V - \omega) u + (K_{a, b} \ast u^2) u
=
0
&\text{in} ~ \mathbb{R}^3;
\\
\|u\|_{\mathscr{L}^2}^2 = \mu,
\end{cases}
\]
where
\[
K_{a, b} (x)
:=
\frac{1}{|x|} \left(
	\frac{4}{3} e^{- b |x|}
	-
	\frac{1}{3} e^{- a |x|}
	-
	1
\right);
\]
$0 \leq a, b \leq \infty$; $V$ denotes a singular potential that vanishes at infinity and the unknowns are
$\omega \in \mathbb{R}$, $u \colon \mathbb{R}^3 \to \mathbb{R}$. This problem is obtained by looking for standing waves of the Schrödinger equation coupled with the nonrelativistic gravitational potential prescribed by a family of fourth-order gravity theories. In this paper, (i) we obtain a complete picture of the existence/nonexistence of ground states of the associated autonomous problem for every possible geometry of $K_{a, b}$, (ii) we obtain conditions that ensure the existence of ground states of the nonautonomous problem when $K_{a, b} \leq 0$ and (iii) we prove that as
\[
(a, b)
\to
(A, B)
\in
\left\{(0, 0), (\infty, \infty), (0, \infty)\right\},
\]
ground states of this problem respectively converge to a ground state of (1) the Schrödinger equation, (2) the Choquard equation and (3) a rescaling of the Choquard equation.

\vspace{1em}
\noindent
\textbf{Keywords.} Yukawa potential, prescribed norm, prescribed mass, mass constraint, constrained energy functional
\end{abstract}

\tableofcontents

\section{Introduction}
This paper is concerned with ground states of the following nonautonomous normalized nonlocal semilinear equation:
\begin{equation}
\label{intro:eqn:nonautonomous}
\begin{cases}
- \Delta u + \parens{V - \omega} u
+
\parens{K_{a, b} \ast u^2} u
=
0
&\text{in} ~ \mathbb{R}^3;
\\
\|u\|_{\mathscr{L}^2}^2 = \mu,
\end{cases}
\end{equation}
where
\[
K_{a, b} (x)
:=
\frac{1}{|x|} \parens*{
	\frac{4}{3} e^{- b |x|}
	-
	\frac{1}{3} e^{- a |x|}
	-
	1
};
\]
$0 \leq a, b \leq \infty$; $V$ denotes a singular potential and we want to solve for
$\omega \in \mathbb{R}$, $u \colon \mathbb{R}^3 \to \real$. More precisely, we set $e^{- 0 |\cdot|} \equiv 1$ and
$e^{- \infty |\cdot|} \equiv 0$.

\sloppy
Problem \eqref{intro:eqn:nonautonomous} is obtained when looking for standing waves of the Schrödinger equation coupled with the nonrelativistic gravitational potential prescribed by a family of fourth-order modified gravity theories indexed by the parameters $a, b$ (see Appendix \ref{physics}). As suggested in \cite[Section 5]{perlickSelfforceElectrodynamicsImplications2015}, a motivation for considering such a family of theories is to obtain a manageable model for the Newtonian gravitational self-force over massive point particles (see \cite[Section V]{poissonGravitationalSelfforce2005}) without the need of renormalization procedures when
$0 \leq a, b < \infty$.

More precisely, we obtain a manageable term for the self-force when $\int \abs{\nabla K_{a, b}}^2 \dif x$ does not diverge. In fact, $K_{a, b}$ has significantly different geometries according to how $a, b$ compare with each other (see Table \ref{intro:table:qualitative} and Appendix \ref{study}). In particular, we highlight the following asymptotic behaviors of $K_{a, b} \parens{x}$ as $\abs{x} \to \infty$.
\begin{itemize}
\item
If $0 = a < b \leq \infty$, then
$K_{a, b} = K_{0, b}$ models an attractive potential that is stronger than the Newtonian potential at long distances because
\begin{equation}
\label{intro:eqn:asy-1}
\abs{x} K_{0, b} \parens{x}
\xrightarrow[\abs{x} \to \infty]{}
- \frac{4}{3}.
\end{equation}
\item
If $0 < a, b \leq \infty$, then $K_{a, b}$ is indistinguishable from the Newtonian potential at long distances in the sense that
\begin{equation}
\label{intro:eqn:asy-2}
\abs{x} K_{a, b} \parens{x} \xrightarrow[\abs{x} \to \infty]{} -1.
\end{equation}
\end{itemize}

\begin{table}[h]
\label{intro:table:qualitative}
\centering
\begin{tabular}{c c c c}
Case
&Sign
&Radial monotonicity
&$\int \abs{\nabla K_{a, b} \parens{x}}^2 \dif x$
\\
\hline
$a = b = 0$
&Identically zero
&-
&0
\\ \\
$0 \leq a \leq 2 b = \infty$
&Negative
&Strictly increasing
&$\infty$
\\ \\
$0 \leq a \leq 2 b < \infty$
&Negative
&Strictly increasing
&$< \infty$
\\ \\
$0 < 2 b < a \leq 4 b < \infty$
&Negative
&Not monotonous
&$< \infty$
\\ \\
$0 < 4 b < a < \infty$
&Sign-changing
&Not monotonous
&$< \infty$
\\ \\
$0 < 4 b < a = \infty$
&Sign-changing
&Not monotonous
&$\infty$
\\ \\
$0 = 4 b < a < \infty$
&Positive
&Strictly decreasing
&$< \infty$
\\ \\
$0 = 4 b < a = \infty$
&Positive
&Strictly decreasing
&$\infty$
\end{tabular}
\caption{Qualitative behavior of $K_{a, b}$.}
\end{table}

We obtain various normalized problems of physical interest by replacing $a, b$ with specific values.
\begin{itemize}
\item
When $a = b = 0$, \eqref{intro:eqn:nonautonomous} becomes the normalized stationary \emph{Schrödinger equation},
\begin{equation}
\label{intro:eqn:normalized_Schroedinger}
\begin{cases}
- \Delta u + \parens{V - \omega} u
=
0;
\\
\|u\|_{\mathscr{L}^2}^2 = \mu.
\end{cases}
\end{equation}
\item
When $a = b = \infty$, \eqref{intro:eqn:nonautonomous} becomes the normalized \emph{Choquard equation},
\begin{equation}
\label{intro:eqn:normalized_Choquard}
\begin{cases}
- \Delta u + \parens{V - \omega} u - \parens*{|\cdot|^{- 1} \ast u^2} u
=
0;
\\
\|u\|_{\mathscr{L}^2}^2 = \mu.
\end{cases}
\end{equation}
\item
When $\parens{a, b} = \parens{0, \infty}$, \eqref{intro:eqn:nonautonomous} becomes a rescaling of the previous problem,
\begin{equation}
\label{intro:eqn:normalized_Choquard-type}
\begin{cases}
- \Delta u + \parens{V - \omega} u
-
\frac{4}{3} \parens*{|\cdot|^{- 1} \ast u^2} u
=
0;
\\
\|u\|_{\mathscr{L}^2}^2 = \mu.
\end{cases}
\end{equation}
\item
When $0 \leq a < 2 b < \infty$, \eqref{intro:eqn:nonautonomous} becomes a normalized stationary equation with a \emph{Bopp--Podolsky-type self-attractive interaction} (see \cite{zhangNormalizedSolutionSchrodinger2025}).
\item
When $0 = b < a = \infty$, \eqref{intro:eqn:nonautonomous} becomes the normalized stationary equation obtained by reducing the \emph{Schrödinger--Maxwell system} (see \cite{benciEigenvalueProblemSchrodingerMaxwell1998}),
\begin{equation}
\label{intro:eqn:normalized_SM}
\begin{cases}
- \Delta u + \parens{V - \omega} u
+
\frac{1}{3} \parens*{|\cdot|^{- 1} \ast u^2} u
=
0;
\\
\|u\|_{\mathscr{L}^2}^2 = \mu.
\end{cases}
\end{equation}
\item
When $0 = b < a < \infty$, \eqref{intro:eqn:nonautonomous} becomes the normalized stationary equation obtained by reducing the \emph{Schrödinger--Bopp--Podolsky system} (see \cite{daveniaNonlinearSchrodingerEquation2019}),
\begin{equation}
\label{intro:eqn:normalized_SBP}
\begin{cases}
- \Delta u + \parens{V - \omega} u
+
\parens{K_{a, 0} \ast u^2} u
=
0;
\\
\|u\|_{\mathscr{L}^2}^2 = \mu,
\end{cases}
\end{equation}
\end{itemize}

The archetypical example of a nonlocal semilinear elliptic equation is the Choquard equation
\begin{equation}
\label{intro:eqn:Choquard}
- \Delta u + u
-
\parens{I_\alpha \ast \abs{u}^p} \abs{u}^{p - 2} u
=
0
~ \text{in} ~
\real^N,
\end{equation}
where
$0 < \alpha < N$,
$
\frac{N + \alpha}{N}
\leq p \leq
\frac{N + \alpha}{\parens{N - 2}^+}
$
and
$
I_\alpha \parens{x}
:=
\abs{x}^{- \parens{N - \alpha}}
$.
The literature concerned with problems related with \eqref{intro:eqn:Choquard} is too vast to list here, so we restrict ourselves to cite the seminal papers \cite{liebExistenceUniquenessMinimizing1977, lionsChoquardEquationRelated1980} and the relatively recent survey about related problems \cite{morozGuideChoquardEquation2017}.

On the other hand, more general normalized nonlocal semilinear elliptic problems of the form
\[
\begin{cases}
- \Delta u + \parens{V - \omega} u
+
\parens{W \ast \abs{u}^p} \abs{u}^{p - 2} u = 0
&\text{in} ~
\real^N;
\\
\norm{u}_{\Lebesgue^2}^2 = \mu
\end{cases}
\]
have received much less attention, especially when $W$ is sign-changing or not homogeneous. In the autonomous case $V \equiv 0$, the recent papers \cite{bhimaniNormalizedSolutionsNonlinear2024, caoVariationalProblemRepulsiveattractive2025} studied the case where $W$ is of the form $I_\alpha - I_\beta$ with
$\alpha \neq \beta$ in $\ooi{0, N}$. In the nonautonomous case
$V \not \equiv 0$, the problem where $W$ is a sum of Riesz potentials with a weakly attractive
$V \in C^1 \parens{\real^N, \oci{- \infty, 0}}$ was considered in \cite{longNormalizedSolutionsCritical2023}.\footnote{We say that $V$ is \emph{weakly attractive potential} when
$V \parens{x} \leq \limsup_{\abs{y} \to \infty} V \parens{y}$
for every $x \in \real^N$.}
Similar nonautonomous problems were also studied in \cite{songNormalizedSolutionsPlanar2025, shenConcentratingNormalizedSolutions2025}.

In a similar direction, there has been an increasing interest about normalized elliptic problems with a local nonlinearity of the form
\begin{equation}
\label{intro:eqn:normalized_local}
\begin{cases}
- \Delta u + \parens{V - \omega} u  = f \parens{\cdot, u}
&\text{in} ~ \real^N;
\\
\norm{u}_{\Lebesgue^2}^2 = \mu.
\end{cases}
\end{equation}
First, let us consider mass-subcritical problems. The influential paper \cite{ikomaStableStandingWaves2020} addressed the existence of ground states of \eqref{intro:eqn:normalized_local} with an autonomous nonlinearity $f \parens{x, u} = f \parens{u}$ and under the hypothesis that $V$ is a continuous weakly attractive negative potential that vanishes at infinity. In \cite{yangNormalizedSolutionsNonlinear2022}, the existence of ground states was later extended for the case of a nonautonomous nonlinearity $f \parens{x, u}$ with a continuous weakly attractive potential $V$. In \cite{alvesNormalizedSolutionsSchrodinger2022}, the authors studied the existence of ground states for the mass-subcritical power nonlinearity $f \parens{u} = \abs{u}^{p - 2} u$ under the hypothesis that $V \colon \real^N \to \coi{0, \infty}$ is continuous and bounded. Problems of the form \eqref{intro:eqn:normalized_local} in the mass-supercritical case are also receiving increasing attention. Inspired by the pioneering paper \cite{jeanjeanExistenceSolutionsPrescribed1997} (see also \cite{bartschNaturalConstraintApproach2017}), the common approach employed in this case is to look for minimizers under the additional constraint given by a Poho\v{z}aev identity. In this context, the paper \cite{dingNormalizedSolutionSchrodinger2022} furnished conditions for the existence of solutions to \eqref{intro:eqn:normalized_local} under certain hypotheses on $V$, including
$
\lim_{\abs{x} \to \infty} V \parens{x}
=
\sup_{x \in \real^N} V \parens{x}
=
0
$.

As exemplified by the cited articles, it is standard to suppose that $V$ is a weakly attractive potential in order to obtain ground states for nonautonomous normalized problems. Influenced by the classical approach in Lions' \cite[Section III]{lionsConcentrationcompactnessPrincipleCalculus1984}, this paper considers a weaker condition (the \emph{energy deficiency condition} \ref{intro:hyp:energy_deficiency}) that ensures the existence of ground states even for certain sign-changing potentials that vanish at infinity.

\subsection{Ground states of the associated autonomous problem}

The first goal of this paper is to study the conditions for existence/nonexistence of ground states of the following autonomous problem associated with \eqref{intro:eqn:nonautonomous}:
\begin{equation}
\label{intro:eqn:autonomous}
\begin{cases}
- \Delta u - \omega u + \parens{K_{a, b} \ast u^2} u
=
0
&\text{in} ~ \real^3;
\\
\|u\|_{\mathscr{L}^2}^2 = \mu.
\end{cases}
\end{equation}
In order to recall the considered notion of ground states, we have to introduce the variational formulation of \eqref{intro:eqn:autonomous}. Problem \eqref{intro:eqn:autonomous} is naturally associated with the \emph{energy functional}
$\E_{a,b}^0 \colon \Sobolev^1 \to \real$
defined as
\[
\E_{a,b}^0 \parens{u}
=
\frac{1}{2} \A \parens{u}
+
\frac{1}{4} \K_{a, b} \parens{u},
\]
where
\[
\A \parens{u} 
:=
\int \abs*{\nabla u \parens{x}}^2 \dif x
\quad \text{and} \quad
\K_{a, b} \parens{u}
:=
\int \int
	K_{a, b} \parens{x - y} u \parens{x}^2 u \parens{y}^2
\dif x \ddif y.
\]
Standard arguments involving the Sobolev embeddings and the Hardy--Littlewood--Sobolev inequality show that these functionals are well defined, so we omit the details. Sometimes, it will also be useful to consider the functional
$\D_c \colon \Sobolev^1 \to \coi{0, \infty}$
defined as
\[
\D_c \parens{u}
=
\int \int
	\frac{e^{- c \abs{x - y}}}{\abs{x - y}}
	u \parens{x}^2 u \parens{y}^2
\dif x \ddif y,
\]
so that
$\K_{a, b} = \frac{4}{3} \D_b - \frac{1}{3} \D_a - \D_0$. In this context, we understand a \emph{ground state} of \eqref{intro:eqn:autonomous} to be a solution to the following minimization problem:
\[
\begin{cases}
\E_{a, b}^0 \parens{u}
=
E_{a, b}^0 \parens{\mu}
:=
\inf_{v \in \Sphere \parens{\mu}}
	\E_{a, b}^0 \parens{v};
\\
u \in \Sphere \parens{\mu},
\end{cases}
\]
where
\[
\Sphere \parens{\mu}
:=
\set*{u \in \Sobolev^1: \norm{u}_{\Lebesgue^2}^2 = \mu}.
\]
In fact, our first result provides the complete picture about existence and nonexistence of ground states of \eqref{intro:eqn:autonomous}.

\begin{thm}
\label{intro:thm:autonomous}
\begin{enumerate}
\item
If $0 \leq a \leq \infty$, $b = 0$ and $\mu > 0$, then
$E_{a, 0}^0 \parens{\mu} = 0$ and \eqref{intro:eqn:autonomous} does not admit ground states.
\item
Suppose that $0 \leq a \leq \infty$, $0 < b \leq \infty$ and
$\mu > 0$.
\begin{enumerate}
\item
Problem \eqref{intro:eqn:autonomous} admits ground states;
\item
its associated ground state energy is negative:
$E_{a, b}^0 \parens{\mu} < 0$;
\item
if $u$ denotes a ground state of \eqref{intro:eqn:autonomous}, then its associated Lagrange multiplier is negative: $\omega < 0$.
\end{enumerate}
\end{enumerate}
\end{thm}

Let us sketch the main ideas involved in the proof. The nonexistence result follows from a usual scaling argument. As for the existence result, if $\parens{u_n}_{n \in \nat}$ denotes a minimizing sequence of $\E_{a, b}^0|_{\Sphere \parens{\mu}}$ and $u_n \rightharpoonup u_\infty$ in $\Sobolev^1$ as $n \to \infty$, then two bad scenarios may occur:
\begin{itemize}
\item
\emph{Vanishing:} $u_\infty \equiv 0$;
\item
\emph{Dichotomy:} $0 < \norm{u_{\infty}}_{\Lebesgue^2}^2 < \mu$.
\end{itemize}
On one hand, we can rule out vanishing with a standard argument involving concentration-compactness. On the other hand, the strategy to rule out dichotomy involves proving the
\emph{Strict Subadditivity Condition}
\begin{equation}
\label{intro:eqn:SSC}
E_{a, b}^0 \parens{\mu}
<
E_{a, b}^0 \parens{\rho}
+
E_{a, b}^0 \parens{\mu - \rho}
~ \text{for every} ~
\mu > 0 ~ \text{and} ~ \rho \in \ooi{0, \mu}.
\end{equation}
In this context, we use the asymptotic behaviors \eqref{intro:eqn:asy-1}, \eqref{intro:eqn:asy-2} together with rescalings to prove that, even if $K_{a, b}$ is sign-changing, the ground state energy function
$\mu \mapsto E_{a, b}^0 \parens{\mu}$
takes negative values and we use this fact to prove \eqref{intro:eqn:SSC}.

\subsection{Ground states of the nonautonomous problem}

We begin by stating the first hypotheses on the external potential $V$.
\begin{enumerate}
[label=$(\mathrm{H}_\arabic*)$]
\item \label{intro:hyp:integrability}
$V \in \Lebesgue^{\frac{3}{2}} + \Lebesgue^\infty$.
\item \label{intro:hyp:vanishes_at_infinity}
$V$ \emph{vanishes at infinity} in the sense that
\[
\mathrm{meas}
	\parens*{\set*{x \in \real^3 : \abs*{V \parens{x}} > c}}
<
\infty
\]
for every $c > 0$.
\end{enumerate}
On one hand, \ref{intro:hyp:integrability} is a classical condition used to ensure that
\[
\Sobolev^1 \ni u
\mapsto
\V \parens{u}
:=
\int V \parens{x} u \parens{x}^2 \dif x
\in
\real
\]
is a well-defined functional. On the other hand, it follows from \ref{intro:hyp:integrability}, \ref{intro:hyp:vanishes_at_infinity} that the functional $\V$ is weakly continuous (see Lemma \ref{prelim:lem:V_is_weakly_continuous}).

\begin{ex}
Conditions \ref{intro:hyp:integrability} and \ref{intro:hyp:vanishes_at_infinity} are satisfied for potentials of the form
\[
V \parens{x}
:=
\sum_{1 \leq k \leq K} \frac{q_k}{\abs{x - x_k}^{\alpha_k}}
\]
with $K \in \nat$; $0 < \alpha_1, \ldots, \alpha_k < 2$;
$x_1, \ldots, x_K \in \real^3$ and
$q_1, \ldots, q_K \in \real$.
\end{ex}

Problem \eqref{intro:eqn:nonautonomous} is naturally associated with the energy functional
$\E_{a, b}^V \colon \Sobolev^1 \to \real$
defined as
\[
\E_{a, b}^V \parens{u}
=
\frac{1}{2} \A \parens{u}
+
\frac{1}{2} \V \parens{u}
+
\frac{1}{4} \K_{a, b} \parens{u}.
\]
As in the case of the autonomous problem, we understand a \emph{ground state} of \eqref{intro:eqn:nonautonomous} to be a solution to the following minimization problem:
\[
\begin{cases}
\E_{a, b}^V \parens{u}
=
E_{a, b}^V \parens{\mu}
:=
\inf_{v \in \Sphere \parens{\mu}}
	\E_{a, b}^V \parens{v};
\\
u \in \Sphere \parens{\mu}.
\end{cases}
\]

Having introduced the variational formulation of \eqref{intro:eqn:nonautonomous}, we proceed to the statement of the last hypothesis on $V$.
\begin{enumerate}
[label=$(\mathrm{H}_\arabic*)$]
\setcounter{enumi}{2}
\item \label{intro:hyp:energy_deficiency}
$E_{a, b}^V \parens{\mu} < E_{a, b}^0 \parens{\mu}$.
\end{enumerate}
Notice that \ref{intro:hyp:integrability}--\ref{intro:hyp:energy_deficiency} are similar to the conditions in \cite[Section III.1]{lionsConcentrationcompactnessPrincipleCalculus1984}. In particular, the \emph{energy deficiency condition} \ref{intro:hyp:energy_deficiency} rules out nonnegative potentials, including the trivial potential $V \equiv 0$. The motivation for this hypothesis is that it ensures the applicability of the usual argument used to rule out vanishing and dichotomy of minimizing sequences in nonautonomous normalized problems (see \cite[Section 3.1]{dinhBlowbehaviorPrescribedMass2020} or \cite[Lemma 4.5]{yangNormalizedSolutionsNonlinear2022}, for instance). We can also consider \ref{intro:hyp:energy_deficiency} a natural generalization of the condition on the Schrödinger operator
$- \Delta + V$ used to ensure that the normalized Schrödinger equation \eqref{intro:eqn:normalized_Schroedinger} admits ground states. Indeed, it follows from Theorem \ref{intro:thm:autonomous} in the case $a = b = 0$ that $E_{0, 0}^0 \parens{\mu} = 0$ and \eqref{intro:eqn:normalized_Schroedinger} does not admit ground states if $V \equiv 0$. On the other hand, it is classical that this problem has ground states if \ref{intro:hyp:integrability}--\ref{intro:hyp:energy_deficiency} are satisfied (see \cite[Theorem 11.5]{liebAnalysis2001}).\footnote{In this context, it is instructive to echo \cite[Remark 1.4]{dinhBlowbehaviorPrescribedMass2020}. Suppose that $a = b = 0$, i.e., $K_{a, b} \equiv 0$. It is clear that \ref{intro:hyp:energy_deficiency} implies $V < 0$ in a set of positive measure. On the other hand, it is not sufficient to suppose that $V < 0$ a.e. to obtain \ref{intro:hyp:energy_deficiency}. Indeed, it follows from the Hardy inequality that if
$- \frac{1}{4 \abs{x}^2} \leq V \parens{x} \leq 0$ for a.e.
$x \in \real^3$, then \ref{intro:hyp:energy_deficiency} is not satisfied. For a sufficient quantitative condition on $V$ that implies \ref{intro:hyp:energy_deficiency}, we refer the reader to \cite[Theorem 11.4.19]{deoliveiraIntermediateSpectralTheory2009}.}

The following result shows that \ref{intro:hyp:energy_deficiency} is stable by small perturbations of the considered potentials and, at least when \eqref{intro:eqn:autonomous} admits ground states, \ref{intro:hyp:energy_deficiency} is weaker than the condition ``$0 \not \equiv V \leq 0$ a.e.''.

\begin{prop}
\label{intro:prop:energy_deficiency}
Suppose that $0 \leq a, b \leq \infty$ and $\mu > 0$.
\begin{enumerate}
\item
\emph{(Nontrivial nonpositive potentials)}
Suppose that \ref{intro:hyp:integrability} is satisfied, \eqref{intro:eqn:autonomous} admits ground states and
$0 \not \equiv V \leq 0$ almost everywhere. Then \ref{intro:hyp:energy_deficiency} is satisfied.
\item
\emph{(Perturbations of energy-deficient potentials)}
Suppose that $V_1$ is a function defined a.e. in $\real^3$ such that \ref{intro:hyp:integrability}--\ref{intro:hyp:energy_deficiency} are satisfied in the case
$V = V_1$. Then there exists $K = K \parens{V_1} > 0$ such that if
$V_2 \in \Lebesgue^{\frac{3}{2}}$,
$V_3 \in \Lebesgue^\infty$
are such that
\[
\mathrm{meas}
	\parens*{\set*{
		x \in \real^3 : \abs*{V_2 \parens{x} + V_3 \parens{x}} > c
	}}
<
\infty
\]
for every $c > 0$ and
\[
K \norm{V_2}_{\Lebesgue^{\frac{3}{2}}}
+
\frac{1}{2}
\norm{V_3}_{\Lebesgue^{\infty}}
\mu
<
E_{a, b}^0 \parens{\mu} - E_{a, b}^{V_1} \parens{\mu},
\]
then \ref{intro:hyp:integrability}--\ref{intro:hyp:energy_deficiency} are satisfied in the case
$V = V_1 + V_2 + V_3$.
If we suppose further that
$\E_{a, b}^{V_1}|_{\Sphere \parens{\mu}}$
admits a ground state $u \in \Sphere \parens{\mu}$, then we can take $K = \frac{1}{2} \norm{u}_{\Lebesgue^6}^2$.
\end{enumerate}
\end{prop}

We make no claim whatsoever that the proposition has considered the optimal class of permitted perturbations. For instance, it is expected that small perturbations by Hardy potentials preserve energy deficiency. In this context, our next result gives sufficient conditions under which the nonautonomous problem \eqref{intro:eqn:nonautonomous} admits ground states when
$K_{a, b} \leq 0$.

\begin{thm}
\label{intro:thm:nonautonomous}
Suppose that $\mu > 0$ and either
\begin{itemize}
\item
$a = b = 0$ (in which case $K_{a, b} \equiv 0$) or
\item
$0 \leq a \leq 4 b$ and $0 < b \leq \infty$ (in which case $K_{a, b} < 0$).
\end{itemize}
Suppose further that \ref{intro:hyp:integrability}--\ref{intro:hyp:energy_deficiency} are satisfied.
\begin{enumerate}
\item
Problem \eqref{intro:eqn:nonautonomous} admits ground states and
\item
if $u$ denotes a ground state of \eqref{intro:eqn:nonautonomous}, then its associated Lagrange multiplier is negative: $\omega < 0$.
\end{enumerate}
\end{thm}

At least when $K_{a, b}$ models nonrepulsive potentials (i.e., $K_{a, b} \leq 0$), the theorem provides a unifying view for the existence of ground states of the aforementioned physical problems obtained for specific values of $a, b$, providing a natural generalization of the well-known criterions for the existence of ground states of the Schrödinger equation (see \cite[Theorem 11.5]{liebAnalysis2001}) and the Choquard equation (see \cite[Theorem III.1]{lionsConcentrationcompactnessPrincipleCalculus1984}).

The theorem only covers the case $K_{a, b} \leq 0$ due to the challenge of disproving the dichotomy of minimizing sequences when $K_{a, b}$ is sign-changing, which we plan to treat in an upcoming paper. The employed method of proof does not clarify whether the considered conditions are, in a certain sense, sharp for the existence of ground states. As such, we invite the reader to reflect about the following question.
\begin{question}
What is the maximal region for the parameters $a, b$ such that \ref{intro:hyp:integrability}--\ref{intro:hyp:energy_deficiency} still ensure the existence of ground states? If these conditions do not ensure the existence of ground states for every
$a, b \in \cci{0, \infty}$, then are there extra conditions on $V$ under which this holds?
\end{question}

\subsection{Asymptotic behavior of ground states}

Our last results are concerned with physically-relevant limit situations of $\parens{a, b}$. The first asymptotic result shows that ground states of \eqref{intro:eqn:nonautonomous} tend to ground states of \eqref{intro:eqn:normalized_Schroedinger} as
$a, b \to 0$.

\begin{thm}
\label{intro:thm:asy-1}
Suppose that \ref{intro:hyp:integrability}, \ref{intro:hyp:vanishes_at_infinity} are satisfied, $\mu > 0$ and consider sequences
$
\set{a_n}_{n \in \nat}, \set{b_n}_{n \in \nat}
\subset
\coi{0, \infty}
$
such that $a_n, b_n \to 0$ as $n \to \infty$. Suppose further that given $n \in \nat$, $u_n$ denotes a ground state of \eqref{intro:eqn:nonautonomous} in the case
$\parens{a, b} = \parens{a_n, b_n}$.
If \ref{intro:hyp:energy_deficiency} is satisfied in the case
$a = b = 0$, then \eqref{intro:eqn:normalized_Schroedinger} has a ground state $u_\infty$ such that
$
\lim_{n \to \infty}
\norm{u_n - u_\infty}_{\Sobolev^1}
=
0
$.
\end{thm}

Next, we show that ground states of \eqref{intro:eqn:nonautonomous} tend to ground states of \eqref{intro:eqn:normalized_Choquard} as $a, b \to \infty$.

\begin{thm}
\label{intro:thm:asy-2}
Suppose that \ref{intro:hyp:integrability}, \ref{intro:hyp:vanishes_at_infinity} are satisfied, $\mu > 0$ and consider sequences
$
\set{a_n}_{n \in \nat}, \set{b_n}_{n \in \nat}
\subset
\oci{0, \infty}
$
such that $a_n, b_n \to \infty$ as $n \to \infty$. Suppose further that given $n \in \nat$, $u_n$ denotes a ground state of \eqref{intro:eqn:nonautonomous} in the case
$\parens{a, b} = \parens{a_n, b_n}$. Then the following implications hold.
\begin{enumerate}
\item
If $V \equiv 0$, then \eqref{intro:eqn:normalized_Choquard} has a ground state $u_\infty$ and there exists
$\set{x_n}_{n \in \nat} \subset \real^3$
such that
$
\lim_{n \to \infty}
\norm*{u_n \parens{\cdot - x_n} - u_\infty}_{\Sobolev^1}
=
0
$.
\item
If \ref{intro:hyp:energy_deficiency} is satisfied in the case
$a = b = \infty$, then \eqref{intro:eqn:normalized_Choquard} has a ground state $u_\infty$ such that
$
\lim_{n \to \infty}
\norm{u_n - u_\infty}_{\Sobolev^1}
=
0
$.
\end{enumerate}
\end{thm}

We then prove that ground states of \eqref{intro:eqn:nonautonomous} tend to ground states of \eqref{intro:eqn:normalized_Choquard-type} as $\parens{a, b} \to \parens{0, \infty}$.

\begin{thm}
\label{intro:thm:asy-3}
Suppose that \ref{intro:hyp:integrability}, \ref{intro:hyp:vanishes_at_infinity} are satisfied, $\mu > 0$ and consider sequences
$\set{a_n}_{n \in \nat} \subset \coi{0, \infty}$,
$\set{b_n}_{n \in \nat} \subset \oci{0, \infty}$
such that $\parens{a_n, b_n} \to \parens{0, \infty}$ as
$n \to \infty$. Suppose further that given $n \in \nat$, $u_n$ denotes a ground state of \eqref{intro:eqn:nonautonomous} in the case $\parens{a, b} = \parens{a_n, b_n}$. Then the following implications hold.
\begin{enumerate}
\item
If $V \equiv 0$, then \eqref{intro:eqn:normalized_Choquard-type} has a ground state $u_\infty$ and there exists
$\set{x_n}_{n \in \nat} \subset \real^3$
such that
$
\lim_{n \to \infty}
\norm*{u_n \parens{\cdot - x_n} - u_\infty}_{\Sobolev^1}
=
0
$.
\item
If \ref{intro:hyp:energy_deficiency} is satisfied in the case
$\parens{a, b} = \parens{0, \infty}$, then \eqref{intro:eqn:normalized_Choquard-type} has a ground state $u_\infty$ such that
$
\lim_{n \to \infty}
\norm{u_n - u_\infty}_{\Sobolev^1}
=
0
$.
\end{enumerate}
\end{thm}

\subsection*{Organization of the text}
\begin{itemize}
\item
Section \ref{prelim} contains technical preliminaries.
\item
Section \ref{autonomous} is concerned with the proof of Theorem \ref{intro:thm:autonomous}.
\item
Section \ref{nonautonomous} provides the proofs of Proposition \ref{intro:prop:energy_deficiency} and Theorem \ref{intro:thm:nonautonomous}.
\item
Section \ref{asy} covers Theorems \ref{intro:thm:asy-1}--\ref{intro:thm:asy-3}.
\item
Appendix \ref{physics} explains the physical motivation for considering \eqref{intro:eqn:nonautonomous}.
\item
Appendix \ref{study} provides a detailed study of the geometry of $K_{a, b}$ and the respective graphs in function of the considered parameters $a, b$.
\end{itemize}

\subsection*{Acknowledgement}

This study was financed, in part, by the São Paulo Research Foundation (FAPESP), Brasil. Process Number \#2024/20593-0.

\section{Preliminaries}
\label{prelim}
The following result is a corollary of the Gagliardo--Nirenberg inequality \cite[Theorem 12.87]{leoniFirstCourseSobolev2017}.
\begin{lem}
\label{prelim:lem:GN}
If $2 \leq s \leq 6$, then there exists $C = C \parens{s} > 0$ such that
\[
\norm{u}_{\Lebesgue^s}
\leq
C
\A \parens{u}^{\frac{3 \parens{s - 2}}{4 s}}
\parens*{
	\int u \parens{x}^2 \dif x
}^{\frac{6 - s}{4 s}}
\]
for every $u \in \Sobolev^1$.
\end{lem}

We recall a particular case of the Hardy--Littlewood--Sobolev inequality \cite[Theorem 4.3]{liebAnalysis2001}.
\begin{lem}
\label{prelim:lem:HLS}
\begin{enumerate}
\item
If $1 < \theta_1, \theta_2 < \infty$ and
$\frac{1}{\theta_1} + \frac{1}{\theta_2} = \frac{5}{3}$,
then there exists $C = C \parens{\theta_1} > 0$ such that
\[
\abs*{
	\int \int
		\frac{f \parens{x} g \parens{y}}{\abs{x - y}}
	\dif x \ddif y
}
\leq
C
\norm{f}_{\Lebesgue^{\theta_1}}
\norm{g}_{\Lebesgue^{\theta_2}}
\]
for every $f \in \Lebesgue^{\theta_1}$ and $g \in \Lebesgue^{\theta_2}$.
\item
Given $t \in \ooi{\frac{1}{2}, \frac{3}{4}}$, there exists
$C = C \parens{t} > 0$ such that
\[
\int \int
	\parens*{
		\frac{u \parens{y}^2}{\abs{x - y}}
	}^{\frac{3 t}{3 - 2 t}}
\dif x \ddif y
\leq
C
\parens*{
	\int \abs*{u \parens{x}}^{2 t} \dif x
}^{\frac{3}{3 - 2 t}}
\]
for every $u \in \Lebesgue^{2 t}$.
\end{enumerate}
\end{lem}

The next result follows directly from \cite[Lemma 2.2]{dinhBlowbehaviorPrescribedMass2020}.

\begin{lem}
\label{prelim:lem:convenient_decomposition}
If \ref{intro:hyp:integrability} and \ref{intro:hyp:vanishes_at_infinity} are satisfied, then we can associate each $\eps > 0$ with a constant $C = C \parens{\eps} > 0$ and functions
$
V_{\frac{3}{2}} = V_{\frac{3}{2}} \parens{\eps}
\in
\Lebesgue^{\frac{3}{2}}
$,
$V_\infty = V_\infty \parens{\eps} \in \Lebesgue^\infty$
such that $V = V_{\frac{3}{2}} + V_\infty$,
$\norm{V_{\frac{3}{2}}}_{\Lebesgue^{\frac{3}{2}}} < \eps$
and
$\norm{V_\infty}_{\Lebesgue^\infty} < C$.
\end{lem}

In view of the previous results, we can obtain a coercivity estimate for $\E_{a, b}|_{\Sphere \parens{\mu}}$ that does not depend on $a, b$.

\begin{lem}
\label{prelim:lem:uniform_estimate}
If \ref{intro:hyp:integrability} and \ref{intro:hyp:vanishes_at_infinity} are satisfied, then there exists $C_2 > 0$ and we can associate each $\eps \in \ooi{0, 1}$ with a $C_1 = C_1 \parens{\eps} > 0$ such that
\[
\E_{a, b} \parens{u}
\geq
\frac{1 - \eps}{2} \A \parens{u}
-
C_1 \mu
-
C_2 \A \parens{u}^{\frac{1}{2}} \mu^{\frac{3}{2}}
\]
for every $\mu > 0$, $u \in \Sphere \parens{\mu}$ and
$a, b \in \cci{0, \infty}$. It follows that given $\mu > 0$ and
$a, b \in \cci{0, \infty}$, $\E_{a, b}|_{\Sphere \parens{\mu}}$ is coercive and $E_{a, b} \parens{\mu} > - \infty$.
\end{lem}
\begin{proof}
~ \paragraph{Association of $\eps > 0$ with $C_1 = C_1 \parens{\eps} > 0$.}
Corollary of Lemma \ref{prelim:lem:convenient_decomposition}.

\paragraph{Existence of the constant $C_2 > 0$.}
It is clear that
$\abs{K_{a, b} \parens{x}} \leq \frac{4}{3 \abs{x}}$
for every $a, b \in \cci{0, \infty}$. As such, it follows from Lemma \ref{prelim:lem:HLS} that there exists $C_2' > 0$ such that
$
\abs{\K_{a, b} \parens{u}}
\leq
C_2' \norm{u}_{\Lebesgue^{\frac{12}{5}}}^4
$
for every $a, b \in \cci{0, \infty}$ and
$u \in \Lebesgue^{\frac{12}{5}}$. In view of this inequality, the result follows from an application of Lemma \ref{prelim:lem:GN}.
\end{proof}

Given $c \in \ooi{0, \infty}$, the function
$
\real^3 \setminus \set{0} \ni x
\mapsto
\frac{e^{- c \abs{x}}}{\abs{x}}
$
is integrable. In particular, we obtain the estimate that follows.
\begin{lem}
Given $c \in \ooi{0, \infty}$ it holds that
\label{prelim:lem:L^4_estimate}
\[
\D_c \parens{v}
=
\int \int
	\frac{e^{- c \abs{x - y}}}{\abs{x - y}}
	v \parens{x}^2
	v \parens{y}^2
\dif x \ddif y
\leq
\frac{4 \pi}{c^2}
\norm{v}_{\Lebesgue^4}^4
\]
for every $v \in \Lebesgue^4$. In view of Lemma \ref{prelim:lem:GN}, there exists $K > 0$ such that
\[
\D_c \parens{v}
\leq
\frac{K}{c^2}
\A \parens{v}^{\frac{3}{2}}
\parens*{\int v \parens{x}^2 \dif x}^{\frac{1}{2}}
\]
for every $c \in \ooi{0, \infty}$ and $v \in \Sobolev^1$.
\end{lem}
\begin{proof}
It follows from Young's inequality that the function
\[
\real^3 \ni x
\mapsto
\int
	\frac{e^{- c \abs{x - y}}}{\abs{x - y}}
	v \parens{y}^2
\dif y
\in
\real
\]
is square-integrable and
\[
\norm*{
	x
	\mapsto
	\int
		\frac{e^{- c \abs{x - y}}}{\abs{x - y}}
		v \parens{y}^2
	\dif y
}_{\Lebesgue^2}
\leq
\parens*{\int \frac{e^{- c \abs{x}}}{\abs{x}}\dif x}
\norm{v}_{\Lebesgue^4}^2
=
\frac{4 \pi}{c^2}
\norm{v}_{\Lebesgue^4}^2.
\]
As such, it follows from Hölder's inequality that
\begin{align*}
\int \int
	\frac{e^{- c \abs{x - y}}}{\abs{x - y}}
	v \parens{x}^2
	v \parens{y}^2
\dif x \ddif y
&\leq
\norm*{
	x
	\mapsto
	\int
		\frac{e^{- c \abs{x - y}}}{\abs{x - y}}
		v \parens{y}^2
	\dif y
}_{\Lebesgue^2}
\norm{v^2}_{\Lebesgue^2}
\\
&\leq
\frac{4 \pi}{c^2}
\norm{v}_{\Lebesgue^4}^4.
\end{align*}
\end{proof}

We recall the following well-known result.

\begin{lem}[{\cite[Theorem 11.4]{liebAnalysis2001}}]
\label{prelim:lem:V_is_weakly_continuous}
If \ref{intro:hyp:integrability}, \ref{intro:hyp:vanishes_at_infinity} hold, then given
$
\set{u_n}_{n \in \nat \cup \set{\infty}}
\subset
\Sobolev^1
$
such that $u_n \rightharpoonup u_\infty$ in $\Sobolev^1$ as
$n \to \infty$, it holds that
$\V \parens{u_n} \to \V \parens{u_\infty}$ as $n \to \infty$.
\end{lem}

The following Brézis--Lieb splitting is proved by arguing as in the proof of \cite[Lemma 2.2]{zhaoExistenceSolutionsSchrodinger2008} or \cite[Proposition 4.2]{boerStandingWavesNonlinear2025}.

\begin{lem}
\label{prelim:lem:BL}
Suppose that $0 \leq m < \infty$. Suppose further that
$\set{u_n}_{n \in \nat \cup \set{\infty}} \subset \Sobolev^1$
are such that $u_n \rightharpoonup u_\infty$ in $\Sobolev^1$ and
$u_n \to u_\infty$ a.e. as $n \to \infty$. Then
\[
\D_m \parens{u_n, u_n}
-
\D_m \parens{u_n - u_\infty, u_n - u_\infty}
\xrightarrow[n \to \infty]{}
\D_m \parens{u_\infty, u_\infty}.
\]
\end{lem}

\section{Ground states of the autonomous problem}
\label{autonomous}

We associate each $u \in \Sobolev^1$ and $\theta > 0$ with the rescaling
\[u_\theta \parens{x} := \theta^3 u \parens{\theta^2 x}.\]
A change of variable shows that
\begin{equation}
\label{aut:eqn:rescaling}
\norm{u_\theta}_{\Lebesgue^2} = \norm{u}_{\Lebesgue^2},
\quad
\A \parens{u_\theta}
=
\theta^4 \A \parens{u}
\quad \text{and} \quad
\K_{a, b} \parens{u_\theta}
=
\theta^2
\K_{\theta^{- 2} a, \theta^{- 2} b} \parens{u}.
\end{equation}
In view of these identities, the following result is a consequence of Lemma \ref{prelim:lem:L^4_estimate}.

\begin{lem}
\label{aut:lem:negative_K}
Suppose that $0 < b \leq \infty$ and
$u \in \Sobolev^1 \setminus \set{0}$.
\begin{enumerate}
\item
If $0 < a \leq \infty$, then
$
\frac{\K_{a, b} \parens{u_\theta}}{\theta^2}
\xrightarrow[\theta \to 0]{}
-
\int \int
	\frac{u \parens{x}^2 u \parens{y}^2}{\abs{x - y}}
\dif x \ddif y
<
0
$.
\item
If $a = 0$, then
$
\frac{\K_{0, b} \parens{u_\theta}}{\theta^2}
\xrightarrow[\theta \to 0]{}
- \frac{4}{3}
\int \int
	\frac{u \parens{x}^2 u \parens{y}^2}{\abs{x - y}}
\dif x \ddif y
<
0.
$
\end{enumerate}
\end{lem}

Next, we prove that the ground state energy is negative when
$b > 0$.

\begin{lem}
\label{aut:lem:negative_energy}
If $0 \leq a \leq \infty$, $0 < b \leq \infty$ and
$\mu > 0$, then $E_{a, b}^0 \parens{\mu} < 0$.
\end{lem}
\begin{proof}
It suffices to prove that there exists $w \in \Sphere \parens{\mu}$ such that $\E_{a, b}^0 \parens{w} < 0$. Consider a fixed
$u \in \Sphere \parens{\mu}$. It follows from \eqref{aut:eqn:rescaling} and Lemma \ref{aut:lem:negative_K} that
\[
\lim_{\theta \to 0}
	\frac{\E_{a, b}^0 \parens{u_\theta}}{\theta^2}
=
\lim_{\theta \to 0}
\frac{\theta^2}{2}
\A \parens{u}
+
\frac{1}{4 \theta^2}
\K_{a, b} \parens{u_\theta}
<
0,
\]
hence the result.
\end{proof}

Our next] goal is to show that the Strict Subadditivity Condition is satisfied.

\begin{lem}
\label{aut:lem:SSC}
If $0 \leq a \leq \infty$ and $0 < b \leq \infty$, then \eqref{intro:eqn:SSC} is satisfied.
\end{lem}
\begin{proof}
In view of Lemmas \ref{prelim:lem:uniform_estimate} and \ref{aut:lem:negative_energy}, we have a well-defined function
$
\ooi{0, \infty} \ni \mu
\mapsto
\frac{E_{a, b}^0 \parens{\mu}}{\mu} \in \ooi{- \infty, 0}
$.
It is classical that it suffices to prove that this function is strictly decreasing. Suppose that $\mu > 0$, $0 < \rho < \mu$ and let $t = \frac{\mu}{\rho} > 1$. Let
$\parens{u_n}_{n \in \nat}$ denote a minimizing sequence of
$\E_{a, b}^0|_{\Sphere \parens{\rho}}$. As
$E_{a, b}^0 \parens{\rho} < 0$, we deduce that
\begin{equation}
\label{aut:lem:SSC:1}
\limsup_{n \to \infty} \K_{a, b} \parens{u_n} < 0.
\end{equation}
By definition,
\[
E_{a, b}^0 \parens{\mu}
\leq
\E_{a, b}^0 \parens*{\sqrt{t} u_n}
-
t \E_{a, b}^0 \parens{u_n}
+
t \E_{a, b}^0 \parens{u_n}
=
\frac{t^2 - t}{4} \K_{a, b} \parens{u_n}
+
t \E_{a, b}^0 \parens{u_n}.
\]
In view of \eqref{aut:lem:SSC:1}, we obtain
\[
E_{a, b}^0 \parens{\mu}
\leq
\frac{t^2 - t}{4}
\parens*{
	\limsup_{m \to \infty}
	\K \parens{u_m}
}
+
t E_{a, b}^0 \parens{\rho}
<
t E_{a, b}^0 \parens{\rho}.
\]
\end{proof}

We proceed to the proof of the theorem.

\begin{proof}[Proof of Theorem \ref{intro:thm:autonomous}]
~\paragraph{Case $0 \leq a \leq \infty$ and $b = 0$: nonexistence of ground states.}

It is clear that $\E_{a, 0}^0 \parens{u} > 0$ for every
$u \in \Sphere \parens{\mu}$. As such, it suffices to prove that
$E_{a, 0}^0 \parens{\mu} = 0$. Consider a fixed
$u \in \Sphere \parens{\mu}$. It follows from \eqref{aut:eqn:rescaling} that
\[
0
\leq
\E_{a, 0}^0 \parens{u_\theta}
\leq
\frac{\theta^4}{2}
\A \parens{u}
+
\frac{\theta^2}{4}
\K_{\infty, 0} \parens{u}
\xrightarrow[\theta \to 0^+]{}
0,
\]
hence the result.

\paragraph{Case $0 \leq a \leq \infty$ and $b > 0$: existence of ground state.}

Let $\parens{v_n}_{n \in \nat}$ denote a minimizing sequence of
$\E_{a, b}^0|_{\Sphere \parens{\mu}}$.
It follows from Lemma \ref{prelim:lem:uniform_estimate} that
$\E_{a, b}^0|_{\Sphere \parens{\mu}}$ is coercive, so
$\parens{v_n}_{n \in \nat}$ is bounded in $\Sobolev^1$. Due to Lemma \ref{aut:lem:negative_energy},
$E_{a, b}^0 \parens{\mu} < 0$.
As such, a standard argument involving concentration-compactness shows that there exists $\set{x_n}_{n \in \nat} \subset \real^3$ such that, up to subsequence,
$\parens{u_n := v_n \parens{\cdot - x_n}}_{n \in \nat}$
does not admit subsequences that converge weakly to zero in
$\Sobolev^1$ (for details, see \cite[p. 275--6]{bellazziniStableStandingWaves2011}).

In particular, there exists
$u_\infty \in \Sobolev^1 \setminus \set{0}$
such that, up to subsequence, $u_n \rightharpoonup u_\infty$ in
$\Sobolev^1$ as $n \to \infty$. If there exists $n_0 \in \nat$ such that $u_n = u_\infty$ for every $n \geq n_0$, then there is nothing to prove. As such, we suppose that, up to subsequence, $u_n \neq u_\infty$ for every $n \in \nat$. In view of the Kondrakov theorem, it also holds that, up to subsequence,
$u_n \to u_\infty$ a.e. as $n \to \infty$.

The norm
$\|\cdot\|_{\Lebesgue^2}$ is weakly lower semicontinuous and $u_\infty \not \equiv 0$, so
$0 < \rho := \norm{u_\infty}_{\Lebesgue^2}^2 \leq \mu$.
We want to show that
\begin{equation}
\label{aut:eqn:0}
\rho = \mu.
\end{equation}
By contradiction, suppose that $\rho < \mu$. Clearly,
\begin{equation}
\label{aut:eqn:1.0}
\E_{a, b}^0 \parens{u_n}
=
\E^0_{a, b} \parens{u_\infty}
+
\E_{a, b}^0 \parens{u_n - u_\infty}
+
\eqref{aut:eqn:1.2}
+
\eqref{aut:eqn:1.3},
\end{equation}
where
\begin{equation}
\label{aut:eqn:1.2}
\frac{1}{2}
\parens*{
	\A \parens{u_n}
	-
	\A \parens{u_\infty}
	-
	\A \parens{u_n - u_\infty}
}
\end{equation}
and
\begin{equation}
\label{aut:eqn:1.3}
\frac{1}{4}
\parens*{
	\K_{a, b} \parens{u_n}
	-
	\K_{a, b} \parens{u_\infty}
	-
	\K_{a, b} \parens{u_n - u_\infty}
}.
\end{equation}
On one hand, the term \eqref{aut:eqn:1.2} tends to zero as
$n \to \infty$ because $u_n \rightharpoonup u_\infty$ in
$\Sobolev^1$ as $n \to \infty$. On the other hand, \eqref{aut:eqn:1.3} tends to zero as $n \to \infty$ due to Lemma \ref{prelim:lem:BL}. We deduce that
\begin{equation}
\label{aut:eqn:2}
\E^0_{a, b} \parens{u_n}
-
\E^0_{a, b} \parens{u_\infty}
-
\E^0_{a, b} \parens{u_n - u_\infty}
\xrightarrow[n \to \infty]{}
0.
\end{equation}
Due to the weak convergence $u_n \rightharpoonup u_\infty$ in
$\Sobolev^1$ as $n \to \infty$, we deduce that
$\norm{u_n - u_\infty}_{\Lebesgue^2} \to \mu - \rho$
as $n \to \infty$. It follows that
\begin{equation}
\label{aut:eqn:3}
\E^0_{a, b} \parens{u_n - u_\infty}
-
\E^0_{a, b}
\parens*{
	\sqrt{\frac{\mu - \rho}{\delta_n}} \parens{u_n - u_\infty}
}
\xrightarrow[n \to \infty]{}
0,
\end{equation}
where
$\delta_n := \norm{u_n - u_\infty}_{\Lebesgue^2}^2$.
In view of \eqref{aut:eqn:2} and \eqref{aut:eqn:3}, we obtain
\[
E_{a, b}^0 \parens{\mu}
\geq
E_{a, b}^0 \parens{\rho}
+
E_{a, b}^0 \parens{\mu - \rho}.
\]
This inequality contradicts Lemma \ref{aut:lem:SSC}, so \eqref{aut:eqn:0} is satisfied.

We just proved that $u_\infty \in \Sphere \parens{\mu}$, so we only have to show that $u_\infty$ is a minimizer of
$\E_{a, b}^0|_{\Sphere \parens{\mu}}$. As $\rho = \mu$, it follows from \eqref{aut:eqn:3} that
$\E^0_{a, b} \parens{u_n - u_\infty} \to 0$
as $n \to \infty$. Finally, the result is a consequence of \eqref{aut:eqn:2}.

\paragraph{Sign of the Lagrange multiplier.}
On one hand,
$\K_{a, b} \parens{u} < 0$ and $\E_{a, b}^0 \parens{u} < 0$
because
$E_{a, b}^0 \parens{\mu} = \E_{a, b}^0 \parens{u} < 0$.
On the other hand, the Nehari identity shows that
\[
\omega \mu
=
\A \parens{u} + \K_{a, b} \parens{u}
=
2 \E_{a, b}^0 \parens{u}
+
\frac{1}{2} \K_{a, b} \parens{u},
\]
hence the result.
\end{proof}

\section{Ground states of the nonautonomous problem}
\label{nonautonomous}
\subsection{Proof of Proposition \ref{intro:prop:energy_deficiency}}
\paragraph{Proof of the first conclusion.}
Let $u$ denote a ground state of \eqref{intro:eqn:autonomous}. We suppose that $0 \not \equiv V \leq 0$ a.e., so
\[
E_{a, b}^V \parens{\mu}
\leq
\E_{a, b}^V \parens{u}
=
\E_{a, b}^0 \parens{u}
+
\frac{1}{2} \V \parens{u}
=
E_{a, b}^0 \parens{\mu}
+
\underbrace{\frac{1}{2} \V \parens{u}}_{< 0}
<
E_{a, b}^0 \parens{\mu},
\]
hence the result.

\paragraph{Proof of the second conclusion.}
Let $\parens{u_n}_{n \in \nat}$ denote a minimizing sequence of
$\E_{a, b}^{V_1}|_{\Sphere \parens{\mu}}$. Clearly,
\begin{align*}
E_{a, b}^{V_1 + V_2 + V_3} \parens{\mu}
&\leq
\E^{V_1}_{a, b} \parens{u_n}
+
\frac{1}{2}
\int
	\parens*{V_2 \parens{x} + V_3 \parens{x}} u_n \parens{x}^2
\dif x
\\
&\leq
\E^{V_1}_{a, b} \parens{u_n}
+
\frac{1}{2}
\norm{V_2}_{\Lebesgue^{\frac{3}{2}}}
\norm{u_n}_{\Lebesgue^6}^2
+
\frac{1}{2}
\norm{V_3}_{\Lebesgue^\infty}
\mu.
\end{align*}
On one hand, it follows from Lemma \ref{prelim:lem:uniform_estimate} that
$\parens{u_n}_{n \in \nat}$ is bounded in $\Sobolev^1$.
On the other hand, $\Sobolev^1 \hookrightarrow \Lebesgue^6$. As such, we deduce that
$
K := \sup_{n \in \nat} \frac{1}{2} \norm{u_n}_{\Lebesgue^6}^2
<
\infty
$.
Therefore,
\begin{align*}
E_{a, b}^{V_1 + V_2 + V_3} \parens{\mu}
&\leq
\E^{V_1}_{a, b} \parens{u_n}
+
K \norm{V_2}_{\Lebesgue^{\frac{3}{2}}}
+
\frac{1}{2}
\norm{V_3}_{\Lebesgue^\infty}
\mu.
\end{align*}
In view of Lemma \ref{prelim:lem:uniform_estimate}, it suffices to take limits to conclude.

\qed

\subsection{Proof of Theorem \ref{intro:thm:nonautonomous}}

The case $a = b = 0$ follows from the classical result \cite[Theorem 11.5]{liebAnalysis2001}.

\paragraph{Existence of ground state.}

Suppose that $\parens{u_n}_{n \in \nat}$ is a minimizing sequence of $\E_{a, b}^V|_{\Sphere \parens{\mu}}$. It suffices to argue as in the proof of Theorem \ref{intro:thm:autonomous} to show that there exists $u_\infty \in \Sobolev^1$ such that, up to subsequence,
$u_n \rightharpoonup u_\infty$ in $\Sobolev^1$ as $n \to \infty$. The norm $\|\cdot\|_{\Lebesgue^2}$ is weakly lower semicontinuous and $u_n \rightharpoonup u_\infty$ in
$\Sobolev^1$ as $n \to \infty$, so
$\norm{u_\infty}_{\Lebesgue^2}^2 \leq \mu$. Let us prove that vanishing does not occur, i.e.,
\begin{equation}
\label{nonaut:eqn:2}
u_\infty \not \equiv 0.
\end{equation}
By contradiction, suppose that $u_\infty \equiv 0$. It follows from Lemma \ref{prelim:lem:V_is_weakly_continuous} that
\[
E_{a, b}^V \parens{\mu}
=
\lim_{n \to \infty} \E_{a, b}^V \parens{u_n}
=
\lim_{n \to \infty} \E_{a, b}^0 \parens{u_n}
\geq
E_{a, b}^0 \parens{\mu}.
\]
This inequality contradicts \ref{intro:hyp:energy_deficiency}, so \eqref{nonaut:eqn:2} holds.

We just proved that $u_\infty \not \equiv 0$, so
$0 < \rho := \norm{u_\infty}_{\Lebesgue^2}^2 \leq \mu$.
If there exists $n_0 \in \nat$ such that $u_n = u_\infty$ for every $n \geq n_0$, then there is nothing to prove. As such, suppose that
$\parens{u_n}_{n \in \nat}$ has a subsequence such that
$u_n \neq u_\infty$ for every $n \in \nat$.
In view of the Kondrakov theorem, it also holds that, up to subsequence,
$u_n \to u_\infty$ a.e. as $n \to \infty$.

By definition,
\begin{align}
\E_{a, b} \parens{u_n}
=
&
\E_{a, b} \parens{u_\infty}
\label{proof:2.1}
\\
&+
\E_{a, b}^0 \parens{u_n - u_\infty}
\label{proof:2.2}
\\
&+
\frac{1}{2}
\parens*{
	\A \parens{u_n}
	-
	\A \parens{u_\infty}
	-
	\A \parens{u_n - u_\infty}
}
\label{proof:2.3}
\\
&+
\frac{1}{2}
\parens*{\V \parens{u_n} - \V \parens{u_\infty}}
\label{proof:2.4}
\\
&+
\frac{1}{4}
\parens*{
	\K_{a, b} \parens{u_n}
	-
	\K_{a, b} \parens{u_\infty}
	-
	\K_{a, b} \parens{u_n - u_\infty}
}.
\label{proof:2.5}
\end{align}
First, consider the term \eqref{proof:2.1}. As $K_{a, b} < 0$, we obtain
\begin{align}
E_{a, b} \parens{\mu}
\leq
\E_{a, b} \parens*{
	\sqrt{\frac{\mu}{\rho}} u_\infty
}
&=
\frac{\mu}{2 \rho}
\A \parens{u_\infty}
+
\frac{\mu}{2 \rho}
\V \parens{u_\infty}
+
\frac{\mu^2}{4 \rho^2}
\K_{a, b} \parens{u_\infty}
\nonumber
\\
&\leq
\frac{\mu}{\rho}
\E_{a, b} \parens{u_\infty}.
\label{nonaut:eqn:3}
\end{align}
Now, consider the term \eqref{proof:2.2}. Let
$\delta_n := \norm{u_n - u_\infty}_{\Lebesgue^2}^2$,
so that $\delta_n \to \mu - \rho < \mu$ as $n \to \infty$.
As argued for the previous term,
\begin{align*}
E_{a, b}^0 \parens{\mu}
\leq
\E_{a, b}^0 \parens*{
	\sqrt{\frac{\mu}{\delta_n}} \parens{u_n - u_\infty}
}
&=
\frac{\mu}{2 \delta_n}
\A \parens{u_n - u_\infty}
+
\frac{\mu^2}{4 \delta_n^2}
\K_{a, b} \parens{u_n - u_\infty}
\\
&\leq
\frac{\mu}{\delta_n}
\E_{a, b}^0 \parens{u_n - u_\infty}
\end{align*}
for sufficiently large $n \in \nat$. That is,
\begin{equation}
\label{nonaut:eqn:4}
\E_{a, b}^0 \parens{u_n - u_\infty}
\geq
\frac{\delta_n}{\mu} E_{a, b}^0 \parens{\mu}.
\end{equation}
Let us show that the terms \eqref{proof:2.3}--\eqref{proof:2.5} tend to zero as $n \to \infty$. Indeed, \eqref{proof:2.3} tends to zero as $n \to \infty$ because $u_n \rightharpoonup u_\infty$ in $\Sobolev^1$ as
$n \to \infty$. The term \eqref{proof:2.4} tends to zero as
$n \to \infty$ due to Lemma \ref{prelim:lem:V_is_weakly_continuous}. The term \eqref{proof:2.5} also tends to zero as $n \to \infty$ due to Lemma \ref{prelim:lem:BL}. In view of \eqref{nonaut:eqn:3}, \eqref{nonaut:eqn:4} and this discussion, we obtain
\begin{equation}
\label{nonaut:eqn:5}
\E_{a, b} \parens{u_n}
\geq
\E_{a, b} \parens{u_\infty}
+
\frac{\delta_n}{\mu} E_{a, b}^0 \parens{\mu}
+
o_n \parens{1}
\geq
\frac{\rho}{\mu} E_{a, b} \parens{\mu}
+
\frac{\delta_n}{\mu} E_{a, b}^0 \parens{\mu}
+
o_n \parens{1}.
\end{equation}

Our next goal is to prove that
\begin{equation}
\label{nonaut:eqn:6}
\norm{u_\infty}_{\Lebesgue^2}^2 = \mu.
\end{equation}
By contradiction, suppose that $\rho < \mu$. It follows from \eqref{nonaut:eqn:5} that
$E_{a, b} \parens{\mu} \geq E_{a, b}^0 \parens{\mu}$.
This contradicts \ref{intro:hyp:energy_deficiency}, so \eqref{nonaut:eqn:6} holds.

We just proved that $u_\infty \in \Sphere \parens{\mu}$, so we only have to show that $u_\infty$ is a minimizer of $\E_{a, b}|_{\Sphere \parens{\mu}}$. As $\rho = \mu$, it follows that $\delta_n \to 0$ as
$n \to \infty$. Therefore, we have
$\E_{a, b} \parens{u_\infty} \leq E_{a, b} \parens{\mu}$
due to \eqref{nonaut:eqn:5}. Finally, the result follows from the definition of $E_{a, b} \parens{\mu}$.

\paragraph{Sign of the Lagrange multiplier.}
It suffices to argue as in the proof of Theorem \ref{intro:thm:autonomous}.
\qed

\section{Asymptotic behavior of ground states}
\label{asy}
\subsection{Case 1: $\parens{a_n, b_n} \to \parens{0, 0}$ as $n \to \infty$}

Consider the function defined as
$f \parens{t} = \frac{1 - e^{- c t}}{t}$
for every $t > 0$, where $c$ denotes a fixed positive number. The following estimate is a consequence of the fact that $f$ is strictly decreasing and $f \parens{0^+} = c$.
\begin{lem}
\label{asy-2:lem:L^2_estimate}
It holds that
\[
\int \int
	\frac{1 - e^{- c \abs{x - y}}}{\abs{x - y}}
	u \parens{x}^2
	u \parens{y}^2
\dif x \ddif y
\leq
c \norm{u}_{\Lebesgue^2}^4
\]
for every $c \in \ooi{0, \infty}$ and $u \in \Lebesgue^2$.
\end{lem}

\begin{proof}[Proof of Theorem \ref{intro:thm:asy-1}]
For the sake of simplicity, we omit the dependence on $V$. It is worth introducing a decomposition of $K_{a_n, b_n}$. Consider the functions
$K_{a_n, b_n}^\pm \colon \real^3 \setminus \set{0} \to \real$
defined as
\[
K_{a_n, b_n}^+ \parens{x}
=
\frac{4 e^{- b_n \abs{x}} - 4}{3 \abs{x}}
\quad \text{and} \quad
K_{a_n, b_n}^- \parens{x}
=
\frac{e^{- a_n \abs{x}} - 1}{3 \abs{x}},
\]
so that $K_{a_n, b_n} = K_{a_n, b_n}^+ - K_{a_n, b_n}^-$.

Let us prove that
\begin{equation}
\label{asy-2:eqn:limsup}
\limsup_{n \to \infty} E_{a_n, b_n} \parens{\mu} \leq E_{0, 0} \parens{\mu}.
\end{equation}
In view of \cite[Theorem 11.5]{liebAnalysis2001} (or Theorem \ref{intro:thm:nonautonomous} in the case $a = b = 0$), we can let
$v_0$ denote a ground state of \eqref{intro:eqn:normalized_Schroedinger}. It follows from the definition of $v_0$ that
\[
E_{0, 0} \parens{\mu}
=
\E_{0, 0} \parens{v_0}
=
\E_{a_n, b_n} \parens{v_0}
-
\frac{1}{4}
\K_{a_n, b_n} \parens{v_0}
\geq
E_{a_n, b_n} \parens{\mu}
-
\frac{1}{4}
\K_{a_n, b_n} \parens{v_0}.
\]
In view of the decomposition of $K_{a_n, b_n}$,
\begin{align*}
\K_{a_n, b_n} \parens{v_0}
&=
\frac{4}{3}
\int \int
	\frac{e^{- b_n \abs{x - y}} - 1}{\abs{x - y}}
	v_0 \parens{x}^2
	v_0 \parens{y}^2
\dif x \ddif y
\\
&-
\frac{1}{3}
\int \int
	\frac{e^{- a_n \abs{x - y}} - 1}{\abs{x - y}}
	v_0 \parens{x}^2
	v_0 \parens{y}^2
\dif x \ddif y.
\end{align*}
Lemma \ref{asy-2:lem:L^2_estimate} implies
$
\abs{\K_{a_n, b_n} \parens{v_0}}
\leq
\frac{\parens{4 b_n + a_n} \mu^2}{3}
\to
0
$
as $n \to \infty$, hence the result.

As
$\set{u_n}_{n \in \nat} \subset \Sphere \parens{\mu}$,
a similar argument shows that
$
\liminf_{n \to \infty} E_{a_n, b_n} \parens{\mu}
\geq
E_{0, 0} \parens{\mu}
$.
In view of \eqref{asy-2:eqn:limsup}, we deduce that
$E_{a_n, b_n} \parens{\mu} \to E_{0, 0} \parens{\mu}$
as $n \to \infty$ and $\parens{u_n}_{n \in \nat}$ is a minimizing sequence of $\E_{0, 0}|_{\Sphere \parens{\mu}}$. At this point, it suffices to argue as in the proof of Theorem \ref{intro:thm:nonautonomous} to finish.
\end{proof}

\subsection{Case 2: $\parens{a_n, b_n} \to \parens{\infty, \infty}$ as $n \to \infty$}

\begin{proof}[Proof of Theorem \ref{intro:thm:asy-2}]
Once again, we omit the dependence on $V$. Let us prove that
\begin{equation}
\label{asy-1:eqn:limsup}
\limsup_{n \to \infty} E_{a_n, b_n} \parens{\mu}
\leq
E_{\infty, \infty} \parens{\mu}.
\end{equation}
In view of Theorems \ref{intro:thm:autonomous} and \ref{intro:thm:nonautonomous}, we can let $v_\infty$ denote a ground state of \eqref{intro:eqn:normalized_Choquard}. It follows from the definition of $v_\infty$ that
\begin{align*}
E_{\infty, \infty} \parens{\mu}
=
\E_{\infty, \infty} \parens{v_\infty}
&=
\E_{a_n, b_n} \parens{v_\infty}
-
\frac{1}{12}
\parens*{
	4 \D_{b_n} \parens{v_\infty}
	-
	\D_{a_n} \parens{v_\infty}
}
\\
&\geq
E_{a_n, b_n} \parens{\mu}
-
\frac{1}{12}
\parens*{
	4 \D_{b_n} \parens{v_\infty}
	-
	\D_{a_n} \parens{v_\infty}
}.
\end{align*}
In view of the Sobolev embedding
$\Sobolev^1 \hookrightarrow \Lebesgue^4$,
the result follows from Lemma \ref{prelim:lem:L^4_estimate}.

Now, we want to show that
\begin{equation}
\label{asy-1:eqn:bounded}
\parens{u_n}_{n \in \nat}
~ \text{is bounded in} ~
\Sobolev^1.
\end{equation}
It follows from \eqref{asy-1:eqn:limsup} that
$\parens{E_{a_n, b_n} \parens{\mu}}_{n \in \nat}$
is bounded. In view of this fact, the result is a consequence of Lemma \ref{prelim:lem:uniform_estimate}.

We proceed to the proof that
\begin{equation}
\label{asy-1:eqn:liminf}
\liminf_{n \to \infty} E_{a_n, b_n} \parens{\mu}
\geq
E_{\infty, \infty} \parens{\mu}.
\end{equation}
It follows from the definition of $u_n$ that
\begin{align*}
E_{a_n, b_n} \parens{\mu}
=
\E_{a_n, b_n} \parens{u_n}
&=
\E_{\infty, \infty} \parens{u_n}
+
\frac{1}{12}
\parens*{
	4 \D_{b_n} \parens{u_n}
	-
	\D_{a_n} \parens{u_n}
}
\\
&\geq
E_{\infty, \infty} \parens{\mu}
+
\frac{1}{12}
\parens*{
	4 \D_{b_n} \parens{u_n}
	-
	\D_{a_n} \parens{u_n}
}.
\end{align*}
In view of the embedding
$\Sobolev^1 \hookrightarrow \Lebesgue^4$,
it follows from \eqref{asy-1:eqn:bounded} that
$\parens{u_n}_{n \in \nat}$ is bounded in $\Lebesgue^4$.
As such, \eqref{asy-1:eqn:liminf} follows from Lemma \ref{prelim:lem:L^4_estimate}.

In view of \eqref{asy-1:eqn:limsup} and \eqref{asy-1:eqn:liminf}, we deduce that $\parens{u_n}_{n \in \nat}$ is a minimizing sequence of $\E_{\infty, \infty}|_{\Sphere \parens{\mu}}$. At this point, there are two possible cases.

\subparagraph{Case 1: $V \equiv 0$.}
In this case, it suffices to argue as in the proof of Theorem \ref{intro:thm:autonomous}.

\subparagraph{Case 2: $V \not \equiv 0$.}
In this case, it suffices to argue as in the proof of Theorem \ref{intro:thm:nonautonomous}.
\end{proof}

\subsection{Case 3: $\parens{a_n, b_n} \to \parens{0, \infty}$ as $n \to \infty$}

\begin{proof}[Proof of Theorem \ref{intro:thm:asy-3}]
We omit the dependence on $V$. In this case, it is obvious that
\begin{equation}
\label{asy-3:eqn:liminf}
\liminf_{n \to \infty} E_{a_n, b_n} \parens{\mu}
\geq
E_{0, \infty} \parens{\mu}.
\end{equation}
Let us prove that
\begin{equation}
\label{asy-3:eqn:limsup}
\limsup_{n \to \infty} E_{a_n, b_n} \parens{\mu}
\leq
E_{0, \infty} \parens{\mu}.
\end{equation}
In view of Theorems \ref{intro:thm:autonomous} and \ref{intro:thm:nonautonomous}, we can let $v_\infty$ denote a ground state of \eqref{intro:eqn:normalized_Choquard-type}. It follows from the definition of $v_\infty$ that
\begin{align*}
E_{0, \infty} \parens{\mu}
=
\E_{0, \infty} \parens{v_\infty}
&=
\E_{a_n, b_n} \parens{v_\infty}
-
\frac{1}{12}
\parens*{
	4 \D_{b_n} \parens{v_\infty}
	-
	\D_{a_n} \parens{v_\infty}
	+
	\D_0 \parens{v_\infty}
}
\\
&\geq
E_{a_n, b_n} \parens{\mu}
-
\frac{1}{12}
\parens*{
	4 \D_{b_n} \parens{v_\infty}
	-
	\D_{a_n} \parens{v_\infty}
	+
	\D_0 \parens{v_\infty}
}.
\end{align*}
On one hand, it suffices to argue as in the proof of Theorem \ref{intro:thm:asy-1} to deduce that
$\D_0 \parens{v_\infty} - \D_{a_n} \parens{v_\infty} \to 0$
as $n \to \infty$. On the other hand, it suffices to argue as in the proof of Theorem \ref{intro:thm:asy-2} to deduce that
$\D_{b_n} \parens{u_n} \to 0$ as $n \to \infty$. We conclude that \eqref{asy-3:eqn:limsup} holds. In view of \eqref{asy-3:eqn:liminf} and \eqref{asy-3:eqn:limsup},  $\parens{u_n}_{n \in \nat}$ is a minimizing sequence of $\E_{0, \infty}|_{\Sphere \parens{\mu}}$.

\subparagraph{Case 1: $V \equiv 0$.}
In this case, it suffices to argue as in the proof of Theorem \ref{intro:thm:autonomous}.

\subparagraph{Case 2: $V \not \equiv 0$.}
In this case, it suffices to argue as in the proof of Theorem \ref{intro:thm:nonautonomous}.
\end{proof}

\appendix
\section{Physical motivation}
\label{physics}
\subsection{Electromagnetic self-force}
\label{physics:electromagnetic}

It is classical that an accelerating electrically charged body loses energy due to electromagnetic radiation (see \cite{perlickSelfforceElectrodynamicsImplications2015, poissonMotionPointParticles2011}, for instance). This phenomenon is usually modeled by means of an \emph{electromagnetic self-force} acting over the aforementioned body.

A naive application of the Maxwell theory predicts that an accelerating charged point particle suffers an infinite self-force. Indeed, a charged point particle at $0 \in \real^3$ in vacuum has an associated electric potential
$\phi \parens{x} = \frac{q}{4 \pi \eps_0 \abs{x}}$,
where $q \in \real \setminus \set{0}$ denotes the charge of the particle and $\eps_0$ denotes the \emph{vacuum permitivity}. In particular, the energy of the ensuing electric field is infinite,
$
\frac{\eps_0}{2}
\int \abs{\nabla \phi \parens{x}}^2 \dif x = \infty
$,
and the divergence of this integral is intimately related with the divergence of the electromagnetic self-force (see \cite[Section 4.1]{perlickSelfforceElectrodynamicsImplications2015}).

The standard procedure to deal with these divergences involves renormalization arguments. In particular, the resulting equation of motion becomes a third order differential equation, called the \emph{Abraham--Lorentz equation}. Nonetheless, there is no consensus that this approach is always the preferred one due to some inconvenient unphysical consequences of third order equations (see \cite[Section 1]{perlickSelfforceElectrodynamicsImplications2015}).

To avoid these problems, another possible approach is to consider alternative electromagnetic theories under which charged point particles do not induce electric fields with infinite energy. One of the simplest examples of this kind of theories is the Bopp--Landé--Thomas--Podolsky (BLTP) theory (see \cite[Section 2.3]{perlickSelfforceElectrodynamicsImplications2015}).

\subsection{Gravitational self-force and a family of fourth-order gravity theories}
In classical physics, the gravitational and electric potentials are described by very similar equations. Indeed, given an electric charge density
$\rho_{\mathrm{Charge}} \colon \real^3 \to \real$
and a mass density
$\rho_{\mathrm{Mass}} \colon \real^3 \to \coi{0, \infty}$,
the respectively generated potentials are described by the laws that follow:
\begin{equation}
\label{intro:eqn:Newton--Maxwell}
- \Delta \phi_{\mathrm{Maxwell}}
=
\eps_0^{- 1} \rho_{\mathrm{Charge}}
\quad \text{and} \quad
\Delta \phi_{\mathrm{Newton}}
=
4 \pi G \rho_{\mathrm{Mass}},
\end{equation}
where $G$ denotes the \emph{Newtonian gravitational constant}. 

Similar to the aforementioned electromagnetic radiation, it follows from general relativity that an accelerating massive point particle also loses energy due to gravitational radiation and thus suffers a reaction by means of a gravitational self-force. Due to the analogy \eqref{intro:eqn:Newton--Maxwell}, the energy of the gravitational field generated by a massive point particle is also infinite and we obtain a divergent self-force once again. In view of this problem, Perlick suggested the consideration of modified gravity theories inspired by BLTP electrodynamics in order to obtain a model where we can avoid these divergences (see \cite[Section 5]{perlickSelfforceElectrodynamicsImplications2015}).

Motivated by this suggestion, we consider a family of modified gravity theories described by Lagrangian densities of the form
\begin{equation}
\label{intro:eqn:Lagragian_density}
L_g
=
\frac{1}{8 \pi G}
\parens*{
	\frac{R}{2}
	+
	\alpha R_{\mu \nu} R^{\mu \nu}
	+
	\beta R^2
},
\end{equation}
where $g$ denotes a Lorentzian metric; $\alpha, \beta$ denote real-valued dimensionless constants; $R_{\mu \nu}$ denotes the local expression of the entries of the Ricci tensor induced by $g$ and $R$ denotes the scalar curvature induced by $g$. In fact, this family of theories was studied in detail in Stelle's seminal paper \cite{stelleClassicalGravityHigher1978}.

Let us recall a few facts about this theory as in \cite[Section 3.5]{schmidtFourthOrderGravity2007}. Physical considerations indicate that we should suppose that $\alpha \geq 0$ and $\alpha + 3 \beta \leq 0$. A comparison with the Proca equation suggests respectively defining the expected masses of the spin-0 and spin-2 graviton as
$a = \frac{1}{\sqrt{- 4 \parens{\alpha + 3 \beta}}}$
and
$b = \frac{1}{\sqrt{2 \alpha}}$.
Finally, in the Newtonian limit of this theory, the gravitational potential generated by a point particle with mass $m$ at the origin becomes given by $\phi \parens{x} = m K_{a, b} \parens{x}$.

\subsection{Coupling the Schrödinger equation with fourth-order gravity}

The following \emph{Schrödinger--Newton system} was originally proposed by Diósi in \cite{diosiGravitationQuantummechanicalLocalization1984} as a model for the wave function of a quantum particle with mass $m$ in a theory where only matter fields are quantized and the gravitational field remains classical at the fundamental level:
\begin{equation}
\label{intro:eqn:SN}
\begin{cases}
-
\frac{\hbar^2}{2 m} \Delta \psi
+
m \phi \psi
=
\iu \hbar \frac{\partial \psi}{\partial t};
\\
\Delta \phi
=
4 \pi G m \abs{\psi}^2
\end{cases}
\quad \text{in} \quad \real^3 \times \real,
\end{equation}
where $\hbar$ denotes the \emph{Planck constant}. It is easy to verify that $|\cdot|^{- 1}$ provides a distributional solution to the problem $- \Delta v = 4 \pi \delta_0$, so the standard procedure to obtain solutions to \eqref{intro:eqn:SN} involves considering the following nonlocal semilinear problem:
\begin{equation}
\label{intro:eqn:reduced_SN}
-
\frac{\hbar^2}{2 m} \Delta \psi
-
G m^2 \parens*{|\cdot|^{- 1} \ast \abs{\psi}^2} \psi
=
\iu \hbar \frac{\partial \psi}{\partial t}.\footnotemark
\end{equation}\footnotetext{For a rigorous justification of this reduction procedure, we refer the reader to \cite[Section 4]{benciEigenvalueProblemSchrodingerMaxwell1998} or
\cite[Section 3]{benciSolitaryWavesNonlinear2002}.}

Motivated by the previous sections, this paper considers the analogous problem obtained by considering the gravitational potential of a point particle obtained in the Newtonian limit of the gravity theory with Lagrangian density \eqref{intro:eqn:Lagragian_density}, that is,
\begin{equation}
\label{physics:eqn:rmSN}
-
\frac{\hbar^2}{2 m} \Delta \psi
+
G m^2 \parens*{K_{a, b} \ast \abs{\psi}^2} \psi
=
\iu \hbar \frac{\partial \psi}{\partial t}.
\end{equation}
Suppose that $\psi$ is of the form
$
\psi \parens{x, t}
=
v \parens{x} \exp \parens{- \frac{\iu \widetilde{\omega} t}{\hbar}}
$
for certain $v \colon \real^3 \to \real$ and
$\widetilde{\omega} \in \real$. Then $\psi$ solves \eqref{physics:eqn:rmSN} if, and only if,
\[
-
\frac{\hbar^2}{2 m} \Delta v
-
\widetilde{\omega} v
+
G m^2 \parens{K_{a, b} \ast v^2} v
=
0.
\]
By setting $\omega = \frac{2 m \widetilde{\omega}}{\hbar^2}$
and $u = \frac{\sqrt{2 G m^3}}{\hbar} v$, we obtain
\[
- \Delta u - \omega u + \parens{K_{a, b} \ast u^2} u
=
0.
\]
This equation inspired the study of the normalized problem \eqref{intro:eqn:nonautonomous}.

\section{A study of the function $K_{a, b}$}
\label{study}

Let $k_{a, b} \colon \ooi{0, \infty} \to \real$ be given by
\[
k_{a, b} \parens{r}
=
\frac{1}{r} \left(
	\frac{4}{3} e^{- b r}
	-
	\frac{1}{3} e^{- a r}
	-
	1
\right),
\]
so that $K_{a, b} \parens{x} = k_{a, b} \parens{\abs{x}}$.

\paragraph{Singular case $\set{a, b} \cap \set{0, \infty} \neq \emptyset$.}
The next table contains a list of the explicit expressions of $k_{a, b}$ in the singular case
$\set{a, b} \cap \set{0, \infty} \neq \emptyset$.

\begin{table}[h]
\centering
\begin{tabular}{c  c  c}
Case &$k_{a, b}(r)$ &Graph
\\
\hline
$a = b = 0$
& $\equiv 0$
& -
\\ \\
$a = 0$, $0 < b < \infty$
& $- \frac{4}{3 r} \parens{1 - e^{- b r}}$
& Fig. \ref{graph2}
\\ \\
$0 = a < b = \infty$
& $- \frac{4}{3 r}$
& Fig. \ref{graph1}
\\ \\
$0 < a < \infty$, $b = \infty$
& $- \frac{1}{r} \parens{1 + \frac{1}{3} e^{- a r}}$
& Fig. \ref{graph1}
\\ \\
$a = b = \infty$
& $- \frac{1}{r}$
& Fig. \ref{graph1}
\\ \\
$a = \infty$, $0 < b < \infty$
& $\frac{1}{r} \parens{\frac{4}{3} e^{- b r} - 1}$
& Fig. \ref{graph7}
\\ \\
$0 = b < a = \infty$
& $\frac{1}{3 r}$
& Fig. \ref{graph9}
\\ \\
$0 = b < a < \infty$
& $\frac{1}{3 r} \parens{1 - e^{- a r}}$
& Fig. \ref{graph8}
\end{tabular}
\caption{Singular cases of $a, b$.}
\end{table}

\paragraph{Case $0 < a \leq 2 b < \infty$.}
Let us prove analytically that the possible geometries of $k_{a, b}$ are illustrated in Figures \ref{graph2}, \ref{graph3}. First, we show that $k_{a, b}$ is nondecreasing. It is clear that
$k_{a, b}' \parens{r} = \frac{f \parens{r}}{3 r^2}$, where
\[
f \parens{r}
:=
r \parens*{a e^{- a r} - 4 b e^{- b r}}
+
3 + e^{- a r} - 4 e^{- b r}.
\]
As such, we only have to prove that $f$ is nonnegative. It follows from calculus that
\[
f' \parens{r}
=
r \parens*{- a^2 e^{- a r} + 4 b^2 e^{- b r}}.
\]
We deduce that $r_0 > 0$ is a critical point of $f$ if, and only if, $e^{\parens{b - a} r_0} = \frac{4 b^2}{a^2}$.
There are two possible cases.
\begin{enumerate}
\item
$a < b$. In this case, $f$ admits a unique critical point on
$\ooi{0, \infty}$,
$r_0 = \frac{1}{b - a} \log \parens*{\frac{4 b^2}{a^2}}$.
As $f'' \parens{0^+} > 0$ and
$\lim_{r \to \infty} f \parens{r} = 3$, we deduce that
$f \parens{r_0} > 0$. In particular, it follows that $f$ is positive.
\item
$b \leq a \leq 2 b$. As $\frac{4 b^2}{a^2} \geq 1$ and $b - a \leq 0$, we deduce that $f$ has no critical points. As $f'' \parens{0^+} \geq 0$, it follows that $f$ only takes positive values.
\end{enumerate}

\paragraph{Case $0 < 2 b < a < \infty$.}
It suffices to argue as before to verify that
$r_1 = \frac{1}{a - b} \log \parens*{\frac{a^2}{4 b^2}} > 0$
is the unique critical point of $k_{a, b}$ on $\ooi{0, \infty}$. A Taylor expansion shows that
\[
k_{a, b} \parens{r}
=
\frac{a - 4 b}{3}
+
\frac{4 b^2 - a^2}{6} r
+
O \parens{r^2}
\]
in a neighborhood of $r = 0$. As such, it is clear that
$k_{a, b} ' \parens{0^+} < 0$ and it suffices to consider the different cases $a < 4 b$,
$a = 4 b$, $4 b < a$ to obtain the geometries illustrated in Figures \ref{graph4}--\ref{graph6}.

\begin{figure}[p]
\centering
\begin{minipage}{0.48\textwidth}
\centering
\includegraphics{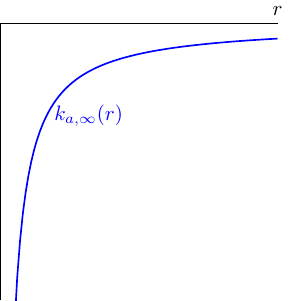}
\caption{$0 \leq a \leq 2 b = \infty$}
\label{graph1}
\end{minipage}
\hfill
\begin{minipage}{0.48\textwidth}
\centering
\includegraphics{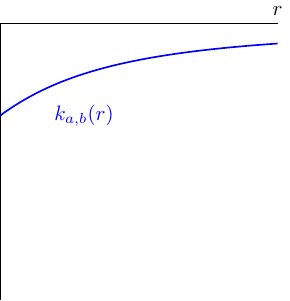}
\caption{$0 \leq a < 2 b < \infty$}
\label{graph2}
\end{minipage}

\vspace{1em}

\begin{minipage}{0.48\textwidth}
\centering
\includegraphics{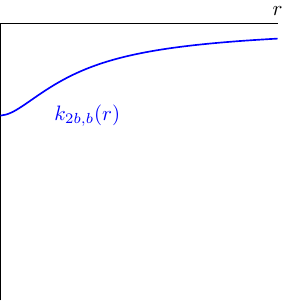}
\caption{$0 < a = 2 b < \infty$}
\label{graph3}
\end{minipage}
\hfill
\begin{minipage}{0.48\textwidth}
\centering
\includegraphics{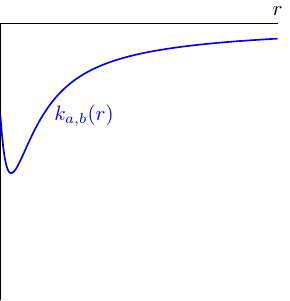}
\caption{$0 < 2 b < a < 4 b < \infty$}
\label{graph4}
\end{minipage}

\vspace{1em}

\begin{minipage}{0.48\textwidth}
\centering
\includegraphics{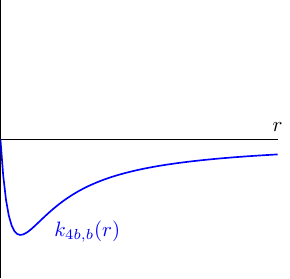}
\caption{$0 < a = 4 b < \infty$}
\label{graph5}
\end{minipage}
\hfill
\begin{minipage}{0.48\textwidth}
\centering
\includegraphics{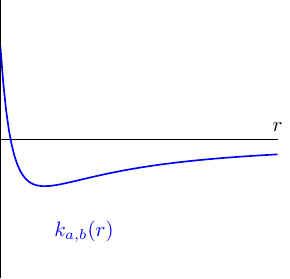}
\caption{$0 < 4 b < a < \infty$}
\label{graph6}
\end{minipage}
\end{figure}

\begin{figure}[t]
\centering
\begin{minipage}{0.48\textwidth}
\centering
\includegraphics{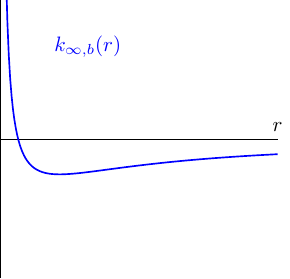}
\caption{$0 < 4 b < a = \infty$}
\label{graph7}
\end{minipage}
\hfill
\begin{minipage}{0.48\textwidth}
\centering
\includegraphics{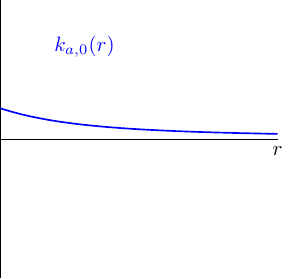}
\caption{$0 = 4 b < a < \infty$}
\label{graph8}
\end{minipage}

\vspace{1 em}

\begin{minipage}{0.48\textwidth}
\centering
\includegraphics{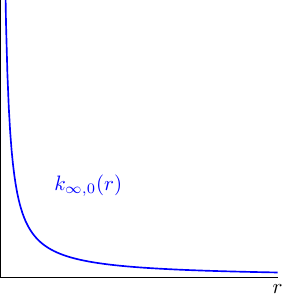}
\caption{$0 = 4 b < a = \infty$}\label{graph9}
\end{minipage}
\end{figure}

\clearpage

\sloppy
\printbibliography

@article{alvesNormalizedSolutionsSchrodinger2022,
  title = {Normalized {{Solutions}} for the {{Schr\"odinger Equations}} with $L^2$-{{Subcritical Growth}} and {{Different Types}} of {{Potentials}}},
  author = {Alves, Claudianor O. and Ji, Chao},
  year = 2022,
  month = may,
  journal = {J Geom. Anal.},
  volume = {32},
  number = {5},
  pages = {165},
  issn = {1050-6926, 1559-002X},
  doi = {10.1007/s12220-022-00908-0},
  urldate = {2025-11-29},
  langid = {english}
}

@article{bartschNaturalConstraintApproach2017,
  title = {A Natural Constraint Approach to Normalized Solutions of Nonlinear {{Schr\"odinger}} Equations and Systems},
  author = {Bartsch, Thomas and Soave, Nicola},
  year = 2017,
  month = jun,
  journal = {J. Funct. Anal.},
  volume = {272},
  number = {12},
  pages = {4998--5037},
  issn = {00221236},
  doi = {10.1016/j.jfa.2017.01.025},
  urldate = {2025-04-17},
  langid = {english}
}

@article{bellazziniStableStandingWaves2011,
  title = {Stable Standing Waves for a Class of Nonlinear {{Schr\"odinger-Poisson}} Equations},
  author = {Bellazzini, Jacopo and Siciliano, Gaetano},
  year = 2011,
  month = apr,
  journal = {Z. Angew. Math. Phys.},
  volume = {62},
  number = {2},
  pages = {267--280},
  issn = {0044-2275, 1420-9039},
  doi = {10.1007/s00033-010-0092-1},
  urldate = {2024-08-08},
  copyright = {http://www.springer.com/tdm},
  langid = {english}
}

@article{benciEigenvalueProblemSchrodingerMaxwell1998,
  title = {An Eigenvalue Problem for the {{Schr\"odinger-Maxwell}} Equations},
  author = {Benci, Vieri and Fortunato, Donato},
  year = 1998,
  month = jun,
  journal = {Topol. Methods Nonlinear Anal.},
  volume = {11},
  number = {2},
  pages = {283},
  issn = {1230-3429},
  doi = {10.12775/TMNA.1998.019},
  urldate = {2025-01-22}
}

@article{benciSolitaryWavesNonlinear2002,
  title = {Solitary Waves of the Nonlinear {{Klein-Gordon}} Equation Coupled with the {{Maxwell}} Equations},
  author = {Benci, Vieri and Fortunato, Donato},
  year = 2002,
  journal = {Rev. Math. Phys.},
  volume = {14},
  number = {4},
  pages = {409--420},
  issn = {0129-055X},
  doi = {10.1142/S0129055X02001168},
  fjournal = {Reviews in Mathematical Physics. A Journal for Both Review and Original Research Papers in the Field of Mathematical Physics},
  mrclass = {35J50 (35J60 35L70 35Q51 35Q60 78A60)},
  mrnumber = {1901222},
  mrreviewer = {Markus Kunze}
}

@article{bhimaniNormalizedSolutionsNonlinear2024,
  title = {Normalized Solutions to Nonlinear {{Schr\"odinger}} Equations with Competing {{Hartree}}-type Nonlinearities},
  author = {Bhimani, Divyang and Gou, Tianxiang and Hajaiej, Hichem},
  year = 2024,
  month = jul,
  journal = {Math. Nachr.},
  volume = {297},
  number = {7},
  pages = {2543--2580},
  issn = {0025-584X, 1522-2616},
  doi = {10.1002/mana.202200443},
  urldate = {2025-08-29},
  langid = {english}
}

@article{boerStandingWavesNonlinear2025,
  title = {Standing Waves for Nonlinear {{Hartree}} Type Equations: Existence and Qualitative Properties},
  shorttitle = {Standing Waves for Nonlinear {{Hartree}} Type Equations},
  author = {B{\"o}er, Eduardo and Moreira Dos Santos, Ederson},
  year = 2025,
  month = jun,
  journal = {Calc. Var. Partial Differential Equations},
  volume = {64},
  number = {5},
  pages = {164},
  issn = {0944-2669, 1432-0835},
  doi = {10.1007/s00526-025-03025-2},
  urldate = {2025-06-06},
  langid = {english}
}

@article{caoVariationalProblemRepulsiveattractive2025,
  title = {Variational Problem with Repulsive-Attractive Kernels and Its Application},
  author = {Cao, Daomin and Jia, Huifang and Luo, Xiao},
  year = 2025,
  month = dec,
  journal = {J. Funct. Anal.},
  volume = {289},
  number = {12},
  pages = {111187},
  issn = {00221236},
  doi = {10.1016/j.jfa.2025.111187},
  urldate = {2025-09-15},
  langid = {english}
}

@article{daveniaNonlinearSchrodingerEquation2019,
  title = {Nonlinear {{Schr\"odinger}} Equation in the {{Bopp}}--{{Podolsky}} Electrodynamics: {{Solutions}} in the Electrostatic Case},
  shorttitle = {Nonlinear {{Schr\"odinger}} Equation in the {{Bopp}}--{{Podolsky}} Electrodynamics},
  author = {{d'Avenia}, Pietro and Siciliano, Gaetano},
  year = 2019,
  month = jul,
  journal = {J. Differential Equations},
  volume = {267},
  number = {2},
  pages = {1025--1065},
  issn = {00220396},
  doi = {10.1016/j.jde.2019.02.001},
  urldate = {2024-07-22},
  langid = {english}
}

@book{deoliveiraIntermediateSpectralTheory2009,
  title = {Intermediate {{Spectral Theory}} and {{Quantum Dynamics}}},
  author = {{de Oliveira}, C{\'e}sar R.},
  year = 2009,
  publisher = {Birkh\"auser Basel},
  address = {Basel},
  doi = {10.1007/978-3-7643-8795-2},
  urldate = {2025-10-28},
  copyright = {http://www.springer.com/tdm},
  isbn = {978-3-7643-8794-5},
  langid = {english}
}

@article{dingNormalizedSolutionSchrodinger2022,
  title = {Normalized Solution to the {{Schr\"odinger}} Equation with Potential and General Nonlinear Term: {{Mass}} Super-Critical Case},
  shorttitle = {Normalized Solution to the {{Schr\"odinger}} Equation with Potential and General Nonlinear Term},
  author = {Ding, Yanheng and Zhong, Xuexiu},
  year = 2022,
  month = oct,
  journal = {J. Differential Equations},
  volume = {334},
  pages = {194--215},
  issn = {00220396},
  doi = {10.1016/j.jde.2022.06.013},
  urldate = {2025-11-30},
  langid = {english}
}

@article{dinhBlowbehaviorPrescribedMass2020,
  title = {Blow-up Behavior of Prescribed Mass Minimizers for Nonlinear {{Choquard}} Equations with Singular Potentials},
  author = {Dinh, Van Duong},
  year = 2020,
  month = jul,
  journal = {Monatsh. Math.},
  volume = {192},
  number = {3},
  pages = {551--589},
  issn = {0026-9255, 1436-5081},
  doi = {10.1007/s00605-020-01387-7},
  urldate = {2025-09-19},
  langid = {english}
}

@article{diosiGravitationQuantummechanicalLocalization1984,
  title = {Gravitation and Quantum-Mechanical Localization of Macro-Objects},
  author = {Di{\'o}si, L.},
  year = 1984,
  month = oct,
  journal = {Physics Letters A},
  volume = {105},
  number = {4-5},
  pages = {199--202},
  issn = {03759601},
  doi = {10.1016/0375-9601(84)90397-9},
  urldate = {2025-09-07},
  copyright = {https://www.elsevier.com/tdm/userlicense/1.0/},
  langid = {english}
}

@article{ikomaStableStandingWaves2020,
  title = {Stable Standing Waves of Nonlinear {{Schr\"odinger}} Equations with Potentials and General Nonlinearities},
  author = {Ikoma, Norihisa and Miyamoto, Yasuhito},
  year = 2020,
  month = apr,
  journal = {Calc. Var. Partial Differential Equations},
  volume = {59},
  number = {2},
  pages = {48},
  issn = {0944-2669, 1432-0835},
  doi = {10.1007/s00526-020-1703-0},
  urldate = {2025-09-19},
  langid = {english}
}

@article{jeanjeanExistenceSolutionsPrescribed1997,
  title = {Existence of Solutions with Prescribed Norm for Semilinear Elliptic Equations},
  author = {Jeanjean, Louis},
  year = 1997,
  month = may,
  journal = {Nonlinear Anal.},
  volume = {28},
  number = {10},
  pages = {1633--1659},
  issn = {0362546X},
  doi = {10.1016/S0362-546X(96)00021-1},
  urldate = {2025-11-30},
  copyright = {https://www.elsevier.com/tdm/userlicense/1.0/},
  langid = {english}
}

@book{leoniFirstCourseSobolev2017,
  title = {A First Course in {{Sobolev}} Spaces},
  author = {Leoni, Giovanni},
  year = 2017,
  series = {Graduate Studies in Mathematics},
  edition = {2},
  volume = {181},
  pages = {xxii+734},
  publisher = {American Mathematical Society},
  address = {Providence, RI},
  doi = {10.1090/gsm/181},
  isbn = {978-1-4704-2921-8},
  mrclass = {46E35 (26Axx 26B30 28A78 46-01)},
  mrnumber = {3726909}
}

@book{liebAnalysis2001,
  title = {Analysis},
  author = {Lieb, Elliott and Loss, Michael},
  year = 2001,
  month = mar,
  series = {Graduate {{Studies}} in {{Mathematics}}},
  edition = {Second},
  volume = {14},
  publisher = {American Mathematical Society},
  address = {Providence, Rhode Island},
  doi = {10.1090/gsm/014},
  urldate = {2024-05-31},
  isbn = {978-0-8218-2783-3},
  langid = {english}
}

@article{liebExistenceUniquenessMinimizing1977,
  title = {Existence and {{Uniqueness}} of the {{Minimizing Solution}} of {{Choquard}}'s {{Nonlinear Equation}}},
  author = {Lieb, Elliott H.},
  year = 1977,
  month = oct,
  journal = {Stud. Appl. Math.},
  volume = {57},
  number = {2},
  pages = {93--105},
  issn = {0022-2526, 1467-9590},
  doi = {10.1002/sapm197757293},
  urldate = {2024-09-26},
  copyright = {http://onlinelibrary.wiley.com/termsAndConditions\#vor},
  langid = {english}
}

@article{lionsChoquardEquationRelated1980,
  title = {The {{Choquard}} Equation and Related Questions},
  author = {Lions, P.L.},
  year = 1980,
  journal = {Nonlinear Anal.},
  volume = {4},
  number = {6},
  pages = {1063--1072},
  issn = {0362-546X},
  doi = {10.1016/0362-546X(80)90016-4}
}

@article{lionsConcentrationcompactnessPrincipleCalculus1984,
  title = {The Concentration-Compactness Principle in the {{Calculus}} of {{Variations}}. {{The}} Locally Compact Case, Part 1.},
  author = {Lions, P.L.},
  year = 1984,
  month = apr,
  journal = {Ann. Inst. H. Poincar\'e C Anal. Non Lin\'eaire},
  volume = {1},
  number = {2},
  pages = {109--145},
  issn = {0294-1449, 1873-1430},
  doi = {10.1016/s0294-1449(16)30428-0},
  urldate = {2024-07-22},
  copyright = {https://www.elsevier.com/tdm/userlicense/1.0/}
}

@article{longNormalizedSolutionsCritical2023,
  title = {Normalized {{Solutions}} to the {{Critical Choquard-type Equations}} with {{Weakly Attractive Potential}} and {{Nonlocal Perturbation}}},
  author = {Long, Lei and Li, Fuyi and Rong, Ting},
  year = 2023,
  month = oct,
  journal = {Z. Angew. Math. Phys.},
  volume = {74},
  number = {5},
  pages = {193},
  issn = {0044-2275, 1420-9039},
  doi = {10.1007/s00033-023-02090-x},
  urldate = {2025-10-22},
  langid = {english}
}

@article{morozGuideChoquardEquation2017,
  title = {A Guide to the {{Choquard}} Equation},
  author = {Moroz, Vitaly and Van Schaftingen, Jean},
  year = 2017,
  month = mar,
  journal = {J. Fixed Point Theory Appl.},
  volume = {19},
  number = {1},
  pages = {773--813},
  issn = {1661-7746},
  doi = {10.1007/s11784-016-0373-1}
}

@incollection{perlickSelfforceElectrodynamicsImplications2015,
  title = {On the {{Self-force}} in {{Electrodynamics}} and {{Implications}} for {{Gravity}}},
  booktitle = {Equations of {{Motion}} in {{Relativistic Gravity}}},
  author = {Perlick, Volker},
  editor = {Puetzfeld, Dirk and L{\"a}mmerzahl, Claus and Schutz, Bernard},
  year = 2015,
  volume = {179},
  pages = {523--542},
  publisher = {Springer International Publishing},
  address = {Cham},
  doi = {10.1007/978-3-319-18335-0_15},
  isbn = {978-3-319-18334-3},
  langid = {english}
}

@inproceedings{poissonGravitationalSelfforce2005,
  title = {The Gravitational Self-Force},
  booktitle = {General {{Relativity}} and {{Gravitation}}},
  author = {Poisson, Eric},
  year = 2005,
  month = nov,
  pages = {119--141},
  publisher = {World Scientific},
  address = {RDS Convention Centre, Dublin},
  doi = {10.1142/9789812701688_0012},
  isbn = {978-981-256-424-5},
  langid = {english}
}

@article{poissonMotionPointParticles2011,
  title = {The {{Motion}} of {{Point Particles}} in {{Curved Spacetime}}},
  author = {Poisson, Eric and Pound, Adam and Vega, Ian},
  year = 2011,
  month = dec,
  journal = {Living Rev. Relativ.},
  volume = {14},
  number = {1},
  pages = {7},
  issn = {2367-3613, 1433-8351},
  doi = {10.12942/lrr-2011-7},
  urldate = {2025-09-22},
  langid = {english}
}

@article{schmidtFourthOrderGravity2007,
  title = {Fourth Order Gravity: Equations, History, and Applications to Cosmology},
  shorttitle = {{{FOURTH ORDER GRAVITY}}},
  author = {Schmidt, Hans-J{\"u}rgen},
  year = 2007,
  month = mar,
  journal = {Int. J. Geom. Methods Mod. Phys.},
  volume = {04},
  number = {02},
  pages = {209--248},
  issn = {0219-8878, 1793-6977},
  doi = {10.1142/S0219887807001977},
  urldate = {2025-09-29},
  langid = {english}
}

@article{shenConcentratingNormalizedSolutions2025,
  title = {Concentrating Normalized Solutions for {{2D}} Nonlocal {{Schrodinger}} Equations with Critical Exponential Growth},
  author = {Shen, Liejun and Squassina, Marco},
  year = 2025,
  month = apr,
  journal = {Electron. J. Differential Equations},
  volume = {2025},
  pages = {34},
  issn = {1072-6691},
  doi = {10.58997/ejde.2025.34},
  urldate = {2025-12-02},
  copyright = {https://creativecommons.org/licenses/by/4.0}
}

@article{songNormalizedSolutionsPlanar2025,
  title = {Normalized {{Solutions}} of {{Planar Schr\"odinger-Poisson System Involving Moser-Trudinger Critical Growth}} and {{Potential}}},
  author = {Song, Xin and Wang, Li},
  year = 2025,
  month = aug,
  journal = {Results Math.},
  volume = {80},
  number = {5},
  pages = {158},
  issn = {1422-6383, 1420-9012},
  doi = {10.1007/s00025-025-02470-x},
  urldate = {2025-11-30},
  langid = {english}
}

@article{stelleClassicalGravityHigher1978,
  title = {Classical Gravity with Higher Derivatives},
  author = {Stelle, K. S.},
  year = 1978,
  month = apr,
  journal = {Gen. Relativ. Gravitation},
  volume = {9},
  number = {4},
  pages = {353--371},
  issn = {0001-7701, 1572-9532},
  doi = {10.1007/BF00760427},
  urldate = {2025-10-08},
  copyright = {http://www.springer.com/tdm},
  langid = {english}
}

@article{yangNormalizedSolutionsNonlinear2022,
  title = {Normalized {{Solutions}} of {{Nonlinear Schr\"odinger Equations}} with {{Potentials}} and {{Non-autonomous Nonlinearities}}},
  author = {Yang, Zuo and Qi, Shijie and Zou, Wenming},
  year = 2022,
  month = may,
  journal = {J. Geom. Anal.},
  volume = {32},
  number = {5},
  pages = {159},
  issn = {1050-6926, 1559-002X},
  doi = {10.1007/s12220-022-00897-0},
  urldate = {2025-11-30},
  langid = {english}
}

@article{zhangNormalizedSolutionSchrodinger2025,
  title = {Normalized Solution to {{Schr\"odinger}} System with {{Bopp}}--{{Podolsky-type}} Self-Attraction},
  author = {Zhang, Qi and Rong, Ting},
  year = 2025,
  month = dec,
  journal = {J. Math. Phys.},
  volume = {66},
  number = {12},
  pages = {121504},
  issn = {0022-2488, 1089-7658},
  doi = {10.1063/5.0283840},
  langid = {english}
}

@article{zhaoExistenceSolutionsSchrodinger2008,
  title = {On the Existence of Solutions for the {{Schr\"odinger}}--{{Poisson}} Equations},
  author = {Zhao, Leiga and Zhao, Fukun},
  year = 2008,
  month = oct,
  journal = {J. Math. Anal. Appl.},
  volume = {346},
  number = {1},
  pages = {155--169},
  issn = {0022247X},
  doi = {10.1016/j.jmaa.2008.04.053},
  urldate = {2024-07-22},
  copyright = {https://www.elsevier.com/tdm/userlicense/1.0/},
  langid = {english}
}

\end{document}